\documentclass[102pt]{article}
\usepackage{hyperref}
\usepackage[english]{babel}
\usepackage[fleqn,reqno]{amsmath}
\usepackage{amsfonts}
\usepackage{amstext}
\usepackage{amsbsy}
\usepackage{amsthm}
\usepackage{amssymb}
\usepackage{url}
\usepackage{graphicx}
\usepackage{bbm}
\usepackage{amsmath}
\usepackage{xcolor}
\usepackage{ulem}
\usepackage{comment}
\usepackage{relsize,exscale}

\newtheorem{pro}{Proposition}
\newtheorem{theorem}{Theorem}
\newtheorem{lemma}[theorem]{Lemma}

\newcommand{\bea}{\begin{eqnarray}}
\newcommand{\eea}{\end{eqnarray}}
\newcommand{\beas}{\begin{eqnarray*}}
\newcommand{\eeas}{\end{eqnarray*}}

\def\[{\left[}
\def\]{\right]}
\def\<{\langle}
\def\>{\rangle}
\def\({\left(}
\def\){\right)}

\def\r{\rho}

\def\N{{\mathbb N}}
\def\R{{\mathbb R}}

\def\Z{{\mathbb Z}}
\def\T{{\mathbb T}}

\begin{document}

\title{\textbf{ Discrete stability estimates for the pressureless Euler-Poisson-Boltzmann equations in the Quasi-Neutral limit}}
	\date{}
\author{Mehdi Badsi\thanks{Nantes Universit\'e, Inria centre de l'universit\'e de Rennes (Mingus team) and Laboratoire de Math\'ematiques Jean Leray UMR CNRS 6310, 2 Chemin de la Houssini\`ere BP 92208, 44322 Nantes Cedex 3, France. 
  (mehdi.badsi@univ-nantes.fr).}
\and Nicolas Crouseilles\footnotemark[3]\thanks{Univ Rennes, Inria centre de l'universit\'e de Rennes (Mingus team) and IRMAR UMR CNRS 6625, F-35042 Rennes, France \& ENS Rennes, France. (nicolas.crouseilles@inria.fr).}}

	\maketitle

\begin{abstract} We propose and study a fully implicit finite volume scheme for the pressureless Euler-Poisson-Boltzmann equations on the one dimensional torus. Especially, we design a consistent and dissipative  discretization of the force term which yields an unconditional energy decay. In addition, we establish a discrete analogue of the modulated energy estimate around constant states with a small velocity. Numerical experiments are carried to illustrate our theoretical results and to assess the accuracy of our scheme. A test case of the literature is also illustrated.
\end{abstract}

\paragraph{keywords:}
pressureless Euler-Poisson-Boltzmann, finite-volume, implicit-scheme, stability estimates, modulated energy, plasma, quasi-neutral limit

\paragraph{MSCcodes:} 65M12, 82D10, 65J15.

    \section{Introduction}
     We consider a simplified model of a uni-dimensional plasma in which ions are cold and electrons have reached a thermodynamical equilibrium. The macroscopic density of the electrons is thus assumed to obey the Maxwell-Boltzmann law \cite{lieberman}. We model this plasma using a fluid approach where at time $t \in \R^{+}$ and at position $x \in \T = \R / \Z$, the unknowns are $\rho_{\varepsilon}(t,x) \geq 0$,  $u_{\varepsilon}(t,x) \in \R$, $-\phi_{\varepsilon}(t,x) \in \R$ which stand respectively for the ions density, the ions mean velocity and the electrostatic potential. In dimensionless unit, they are assumed to satisfy the pressureless Euler-Poisson-Boltzmann equations posed on $(0,T] \times \T:$
	\begin{equation} \label{EPB}
\begin{cases}
\: \partial_t \rho_{\varepsilon} + \partial_x(\rho_{\varepsilon} u_{\varepsilon} ) =  0,  \\
   \:  \partial_{t}( \rho_{\varepsilon} u_{\varepsilon} ) + \partial_{x}( \rho_{\varepsilon} u_{\varepsilon}^2 ) = \rho_{\varepsilon} \partial_{x} \phi_{\varepsilon}, \\
\: \varepsilon^2 \partial_{xx}\phi_{\varepsilon} + e^{-\phi_{\varepsilon}}  = \rho_{\varepsilon}, 
\end{cases}
\end{equation}
where $T >0$ is a time horizon and $\varepsilon > 0$ is a physical parameter called the Debye length.
The system \eqref{EPB} is supplemented with an initial condition 
\begin{align}
    &\rho_{\varepsilon}(0,x) = \rho_{\varepsilon}^{\textnormal{ini}}(x), \quad  u_{\varepsilon}(0,x) = u_{\varepsilon}^{\textnormal{ini}}(x), 
\end{align}
where $\rho_{\varepsilon}^{\textnormal{ini}} : x \in \T \longrightarrow \R^{+}$ and $u_{\varepsilon}^{\textnormal{ini}} : x \in \T \longrightarrow \R$ are given functions.
The existence of local in time strong solutions to \eqref{EPB} has been rigorously established in \cite{Lannes2013} in the case of the whole space. The proof can be adapted to the case of the torus. More precisely, provided $(\rho_{\varepsilon}^{\textnormal{ini}},u_{\varepsilon}^{\textnormal{ini}}) \in H^{s}(\T) \times H^{s+1}(\T)$ with $s > \frac{1}{2}$  and $\underset{ x \in \T}{\textnormal{ess}} \inf \rho_{\varepsilon}^{\textnormal{ini}} > 0$ there exists $T_{\varepsilon} > 0$ and a unique strong solution $(\rho_{\varepsilon},u_{\varepsilon},\phi_{\varepsilon})$ with the regularity $(\rho_{\varepsilon},u_{\varepsilon}) \in \mathrm{C}^{0}\big( [0,T_{\varepsilon}];H^{s}(\T) \times H^{s+1}(\T) \big) \cap \mathrm{C}^{1}([0,T_{\varepsilon}]; H^{s-1}(\T) \times H^{s}(\T) )$ and $\phi_{\varepsilon} \in \mathrm{C}^{0}([0,T_{\varepsilon}];H^{s+2}(\T)) \cap \mathrm{C}^{1}([0,T_{\varepsilon}];H^{s+1}(\T)) $ and such that  $\underset{ x \in \T}{\textnormal{ess}} \inf \rho_{\varepsilon}(t,x) > 0$ for $t \in [0,T_{\varepsilon}]$. The study of the quasi-neutral limit $\varepsilon \rightarrow 0$ for the strong solutions has been carried by Pu and Guo in \cite{pu,puguo}. Note that in the case of the Euler-Poisson-Boltzmann equations, the quasi-neutral limit has been studied by Cordier and Grenier in \cite{grenier}.  When $\varepsilon \rightarrow 0$
we formally expect the solution $(\rho_{\varepsilon},u_{\varepsilon},\phi_{\varepsilon})$ to converge towards a solution to the isothermal Euler equations:
\begin{equation} \label{EPB0}
\begin{cases}
\partial_t \rho_0 + \partial_x(\rho_0 u_0 ) =0,  \\
\partial_t (\rho_{0}u_0) + \partial_x\Big( \rho_{0} u_0^2  +e^{-\phi_{0}} \Big)= 0,\\
\rho_{0} = e^{-\phi_{0}}.
\end{cases}
\end{equation}
The system \eqref{EPB0} is hyperbolic symmetrizable. Thus, the existence of local in time strong solutions is an application of the Kato-Lax-Friedrichs theory for symmetric hyperbolic systems \cite{Kato,Benzoni-Serre}. One difficulty in the study of the quasi-neutral limit (see \cite{puguo}) consists in establishing thanks to high order energy estimates that the local time of existence $T_{\varepsilon}$ for \eqref{EPB} does not shrink when $\varepsilon \rightarrow 0$ (that is $\underset{ \varepsilon \rightarrow 0}\liminf \: T_{\varepsilon} > 0$) so that there exists a time $T > 0$ on which both system \eqref{EPB} and \eqref{EPB0} live. 

As far as the numerical approximation of solutions to  \eqref{EPB} is concerned, several works propose finite-difference or finite-volume schemes for plasma fluid models, which are linearly stable in the asymptotic $\varepsilon \rightarrow 0$ and formally converge as $\varepsilon \rightarrow 0$  see for instance \cite{degond_b,alvarez,degond_ap}. Precisely on the pressureless Euler-Poisson-Boltzmann equations \eqref{EPB}, a recent work is \cite{ arun} where the authors study a semi-implicit finite volume scheme based on the so-called staggered discretization studied  in \cite{herbin,eymard} for the compressible Euler equations and the compressible Stokes equations. The authors prove a discrete energy estimate for their scheme using a stabilization term. This stabilization term formally vanishes when the time step of the scheme tends to zero but is formally inconsistent when the time step is fixed and the mesh size tends to zero.  A proof of convergence (up to a subsequence) of the discrete scheme in the limit $\varepsilon \rightarrow 0$ is given thanks to a finite dimensional argument which, in fine, boils down to an application of the Bolzano-Weierstrass theorem. The question of the convergence rate when $\varepsilon \rightarrow 0$ and its dependance  with respect to the dimension of the problem is so far an open problem.

Our main focus in this work is twofold: 
firstly, we prove discrete stability estimates for a fully-implicit finite-volume scheme for \eqref{EPB} which as a by product enables us to prove existence for the scheme. Our formalism also uses a staggered discretization  which enables us to establish a kinetic energy balance somehow similar to Lemma 3.1 in \cite{herbin}. What is new with respect to the existing literature, is a consistent space discretization of the force term in the momentum equation of \eqref{EPB} which is compatible with the discrete continuity equation and leads to an unconditional energy decay. Secondly, we propose a discrete analogue of the modulated energy approach \cite{GVHKR, ingrid, pu, puguo} to establish non linear stability for the constant solutions of the system \eqref{EPB}. Especially, it provides a discrete quantitative stability estimate even when $\varepsilon \rightarrow 0$ for well-prepared initial data. The consistency analysis (at fixed $\varepsilon$) of our scheme based on standard assumptions of the litterature is to a certain extent classical, and is as a matter of fact, omitted.

The plan of this work is as follows. In Section \eqref{sec:conservation} we establish the conservation properties of \eqref{EPB}. In Section \eqref{sec:modulated_energy} we recall the modulated energy estimates. Section \eqref{sec:discrete_scheme} defines the numerical scheme to approximate the solutions to \eqref{EPB}. We establish the discrete energy estimates in \eqref{sec:discrete_energy_estimate} and prove the existence of the scheme. Then in section \eqref{sec:discrete_modulated_energy}, we establish the stability of  constant states with a small velocity using the discrete modulated energy. Eventually, we illustrate our results and discuss the numerical accuracy of our scheme in Section \eqref{sec:numerical_experiment}.

\subsection{Conservation properties}\label{sec:conservation}
In the sequel we consider strong solutions to \eqref{EPB} and \eqref{EPB0} which are both defined on $[0,T]$ and such that their respective density is a positive function on $[0,T] \times \T.$
We first establish the conservations of the system \eqref{EPB}.
\begin{pro}
\label{prop_nrj}{(Conservations)} Let $(\rho_{\varepsilon},u_{\varepsilon},\phi_{\varepsilon})$ be a strong solution  to \eqref{EPB} on $[0,T]$ with $\rho_{\varepsilon} > 0.$  Then we have for $t \in [0,T]:$
\begin{align}
    &\int_{\T} \rho_{\varepsilon}(t,x) d x = \int_{\T} \rho_{\varepsilon}(0,x) d x , \label{mass_con} \\
    &\int_{\T} (\rho_{\varepsilon} u_{\varepsilon})(t,x) d x = \int_{\T} (\rho_{\varepsilon} u_{\varepsilon})(0,x) d x, \label{mom_con}\\
    &{\cal H}(t) = {\cal H}(0) \label{nrj_cons},
\end{align}
where $\mathcal{H}$ is the total energy given by
\begin{equation}
\label{def:H}
{\cal H}(t) =  \int_{\T} \rho_{\varepsilon} \frac{u_{\varepsilon}^2}{2} dx + \int_{\T} \frac{\varepsilon^2}{2} \big |\partial_x \phi_{\varepsilon} \big| ^2 dx  + \int_{\T} h(\phi_{\varepsilon})  dx, \: t \in [0,T],
\end{equation}
where 
\begin{align}\label{def : h}
   h:s \in \R \longmapsto -(s+1)e^{-s}.
\end{align} 
\end{pro}

\begin{proof} The conservation of mass \eqref{mass_con} is readily obtained by integration in space of the continuity equation using the periodicity. As for the total momentum conservation \eqref{mom_con}, we integrate in space the momentum equation and  use the periodicity. It yields using the Poisson equation 
$$
\int_{\T} \partial_{t} (\rho_{\varepsilon} u_{\varepsilon})(t,x) d x = \int_{\T} \rho_{\varepsilon} \partial_{x} \phi_{\varepsilon} = \int_{\T} ( \varepsilon^2 \partial_{xx} \phi_{\varepsilon} + e^{-\phi_{\varepsilon}} ) \partial_{x} \phi_{\varepsilon} dx = \int_{\T}  \partial_{x} \Big( \frac{\varepsilon^2}{2} \vert \partial_{x} \phi_{\varepsilon} \vert^2  - e^{-\phi_{\varepsilon}} \Big) dx = 0.
$$
So we get $\displaystyle \frac{d}{dt} \int_{\T} (\rho_{\varepsilon} u_{\varepsilon})(t,x) dx =0$ for $t \in [0,T]$ and thus \eqref{mom_con}.
We prove the energy conservation \eqref{nrj_cons}.
We multiply the momentum equation by $u_{\varepsilon}$ to get 
$
u_{\varepsilon} \cdot \big( \partial_{t} (\rho_{\varepsilon} u_{\varepsilon}) + \partial_{x}(\rho_{\varepsilon} u_{\varepsilon}^2) \big)  = \rho_{\varepsilon} u_{\varepsilon} \cdot \partial_{x} \phi_{\varepsilon}.
$
Then we re-write the first term as a total derivative plus a residual term. We have,
\begin{align}\label{toto}
\partial_{t}( \rho_{\varepsilon} u_{\varepsilon}^2) + \partial_{x}\big( \rho_{\varepsilon} u_{\varepsilon}^{3} \big)
= \rho_{\varepsilon} u_{\varepsilon}  \big( \partial_{t}u_{\varepsilon} +  \partial_{x}\frac{u_{\varepsilon}^2}{2} \big) + \rho_{\varepsilon} u_{\varepsilon}  \partial_{x} \phi_{\varepsilon}.
\end{align}
Thanks to the momentum and the continuity equation, we have
$ \partial_{t}u_{\varepsilon} +  \partial_{x}\frac{u_{\varepsilon}^2}{2} =  \partial_{x} \phi_{\varepsilon}. 
$
Plugging this relation in \eqref{toto} we get
\begin{align} \label{booba}
\partial_{t}\big(\rho_{\varepsilon} \frac{u_{\varepsilon}^2}{2} \big) + \partial_{x}\big( \rho_{\varepsilon} \frac{u_{\varepsilon}^{3}}{2} \big) 
= \rho_{\varepsilon} u_{\varepsilon}  \partial_{x} \phi_{\varepsilon}.
\end{align}
We then integrate in space \eqref{booba} and use the periodicity to get
\begin{align} \label{tata}
    \frac{d}{dt} \int_{\T} \rho_{\varepsilon} \frac{u_{\varepsilon}^2}{2}
    = \int_{\T} \rho_{\varepsilon} u_{\varepsilon}  \partial_{x} \phi_{\varepsilon} dx = -\int_{\T} \partial_{x} (\rho_{\varepsilon} u_{\varepsilon}) \phi_{\varepsilon} dx = \int_{\T} (\partial_{t} \rho_{\varepsilon}) \phi_{\varepsilon} dx
 \end{align}
where we used the continuity equation for the last equality.
Besides, the Poisson equation gives 
 \begin{align}
 \int_{\T} (\partial_{t} \rho_{\varepsilon}) \phi_{\varepsilon} dx = \int_{\T}\Big[ \varepsilon^2 \partial_{t} (\partial_{xx} \phi_{\varepsilon})\phi_{\varepsilon}  + \partial_{t}(e^{-\phi_{\varepsilon}}) \phi_{\varepsilon} \Big]dx = \int_{\T} \Big[\varepsilon^2 \partial_{x}( \partial_{t} \partial_{x} \phi_{\varepsilon}) \phi_{\varepsilon}  - (\partial_{t} \phi_{\varepsilon}) e^{-\phi_{\varepsilon}} \phi_{\varepsilon} \Big]dx.
 \end{align}
 Using an integration by parts for the first term and the definition of the function $h$ we eventually obtain
 \begin{align}\label{titi}
     \int_{\T} (\partial_{t} \rho_{\varepsilon}) \phi_{\varepsilon} dx = -\frac{d}{dt} \int_{\T} \frac{\varepsilon^2}{2} \vert \partial_{x} \phi_{\varepsilon} \vert^2 dx -\frac{d}{dt}\int_{\T} h(\phi_{\varepsilon}) dx.
 \end{align}
 Gathering \eqref{titi} with \eqref{tata} yields $\frac{d}{dt} {\cal H}(t) = 0$ for $t \in [0,T]$ and thus \eqref{nrj_cons}.
\end{proof}

\subsection{The modulated energy estimate} \label{sec:modulated_energy}

Following \cite{Han-Kwan01082011}, for a strong solution $(\rho_{\varepsilon},u_{\varepsilon},\phi_{\varepsilon})$ to \eqref{EPB} and a strong solution $(\rho_{0},u_{0},\phi_{0})$ to \eqref{EPB0} both defined on $[0,T]$, we define the modulated energy around $(\rho_{0},u_{0},\phi_{0})$ at time $t \in [0,T]$ by
\begin{equation} 
\label{def:E}
\mathcal{E}(t) := \int_{\T} \rho_{\varepsilon} \frac{(u_{\varepsilon}-u_{0})^2}{2} dx + \int_{\T} \frac{\varepsilon^2}{2} \vert \partial_{x} \phi_{\varepsilon} \vert^2 dx 
+ \int_{\T} \Big[\tilde{h}(e^{-\phi_{\varepsilon}}) -\Big( \tilde{h}(\rho_{0}) + \tilde{h}'(\rho_{0})(e^{-\phi_{\varepsilon}} - \rho_{0}) \Big)\Big] dx,
\end{equation}
where  $\tilde{h} : \R^{+}_{\star} \rightarrow \R$ is the Boltzmann entropy function given by
\begin{equation} \label{def:h_tilde}
    \forall \psi > 0, \: \tilde{h}(\psi) := h(-\log(\psi)) = \psi \log(\psi) - \psi.
\end{equation}
We point out that $\mathcal{E}$ is a non negative functional since it is the sum of three non negative functionals. The fact that the last term is non negative is due to the fact that the function $\tilde{h}$ is convex. Provided $\|\rho_{\varepsilon} \|_{L^{\infty}_{t,x}}$ and $\| \frac{1}{\rho_{\varepsilon}} \|_{L^{\infty}_{t,x}}$ are uniformly bounded in $\varepsilon$, we have $\mathcal{E}(t) \gtrsim \|u_{\varepsilon}(t)-u_{0}(t)\|^2_{L^2(\T)} + \| e^{-\phi_{\varepsilon}(t)}- \rho_{0}(t) \|_{L^{2}(\T)}$ where the constant in the inequality is independent of $\varepsilon.$ Thus $\mathcal{E}(t)$ yields a control of the distance in $L^{2}(\T)$ between $(\rho_{\varepsilon},u_{\varepsilon},\phi_{\varepsilon})$ and $(\rho_{0},u_{0},\phi_{0})$ at time $t$.
Expanding the terms of \eqref{def:E}, we have the decomposition for $t \in [0,T]:$
\begin{align} \label{decomp_E}
    &\mathcal{E}(t) = {\cal H}(t) + \mathcal{E}_{kin}(t) 
    - \mathcal{E}_{int}(t),
\end{align}
with ${\cal H}(t)$ given by \eqref{def:H} and 
\begin{align}
& \mathcal{E}_{kin}(t) = \int_{\T} \rho_{\varepsilon} \left( \frac{u_{0}^2}{2} -u_{\varepsilon} u_{0} \right)dx,\\
& \mathcal{E}_{int}(t) = \int_{\T} \Big(\tilde{h}(\rho_{0}) + \tilde{h}'(\rho_{0})(e^{-\phi_{\varepsilon}} - \rho_{0}) \Big)dx.
\end{align}
A simple identity that will be used in the modulated energy estimate is the following.
\begin{lemma} \label{key_identity} Let $(\rho_{\varepsilon},u_{\varepsilon},\phi_{\varepsilon})$ be a strong solution to \eqref{EPB} on $[0,T]$ with $\rho_{\varepsilon} > 0.$ Then for every $\psi \in \mathrm{C}^{1}([0,T] \times \T)$:
\begin{equation}
    \rho_{\varepsilon}  (\partial_{t} + u_{\varepsilon} \partial_{x})(\psi) = \partial_{t}(\rho_{\varepsilon} \psi) + \partial_{x}(\rho_{\varepsilon} u_{\varepsilon} \psi).
\end{equation}
\begin{proof} A direct computation yields
$$
\rho_{\varepsilon}  \big(\partial_{t} + u_{\varepsilon} \partial_{x}\big)(\psi)   = \rho_{\varepsilon} \partial_{t} \psi + \rho_{\varepsilon} u_{\varepsilon} \partial_{x} \psi = \partial_{t} (\rho_{\varepsilon} \psi) + \partial_{x}(\rho_{\varepsilon} u_{\varepsilon} \psi) - \psi \big( \partial_{t}\rho_{\varepsilon} + \partial_{x} (\rho_{\varepsilon} u_{\varepsilon}) \big).
$$
It yields the claim thanks to the continuity equation.
\end{proof}
\end{lemma}
The main quantitative stability estimate is stated in the following proposition.
\begin{pro}{(Modulated energy estimate)}  Let $(\rho_{\varepsilon},u_{\varepsilon},\phi_{\varepsilon})$ be a strong solution to \eqref{EPB} on $[0,T]$ with $\rho_{\varepsilon} > 0$ and $(\rho_{0},u_{0},\phi_{0})$ be a strong solution  to \eqref{EPB0}   on $[0,T]$ with $\rho_{0} > 0$. Then we have for $t \in [0,T],$
\begin{align} \label{modulated_energy_estimate}
\mathcal{E}(t) \leq \mathcal{E}(0)e^{2\|\partial_x u_0\|_{L_{t,x}^{\infty}}t}
+ \varepsilon^2 \int_{0}^{t} \mathcal{L}(\tau) e^{2(t-\tau)\|\partial_x u_0\|_{L_{t,x}^{\infty}} } d\tau. 
\end{align}
where ${\cal L}(\tau) =\int_{\T} \partial_{txx} \phi_\varepsilon \log \rho_0 dx - \int_{\T} \partial_{xx}\phi_\varepsilon u_0 \partial_x \log \rho_0 dx$. \\
\end{pro}

\begin{proof} Thanks to the energy conservation \eqref{nrj_cons} and \eqref{decomp_E}, we have for $t \in [0,T]$, 
\begin{align} \label{dt_nrj_mod}
    \frac{d}{dt} \mathcal{E}(t) = \frac{d}{dt} \mathcal{E}_{kin}(t) 
    - \frac{d}{dt} \mathcal{E}_{int}(t).
\end{align}
We shall now estimate each term. For the first term we have,
$$
\frac{d}{dt} \mathcal{E}_{kin}(t) = \int_{\T} \partial_{t} \rho_{\varepsilon} \Big( \frac{u_{0}^2}{2} - u_{\varepsilon} u_{0} \Big) dx + \int_{\T} \rho_{\varepsilon} \partial_{t} \Big( \frac{u_0^2}{2} - u_{\varepsilon} u_{0} \Big)dx.
$$
Using the continuity equation $\partial_{t} \rho_{\varepsilon} + \partial_{x} (\rho_{\varepsilon} u_{\varepsilon}) = 0$ and an integration by parts we obtain
$$
\frac{d}{dt} \mathcal{E}_{kin}(t) = \int_{\T} \rho_{\varepsilon} u_{\varepsilon} \partial_{x}\Big( \frac{u_0^2}{2} - u_{\varepsilon} u_{0} \Big)dx + \int_{\T} \rho_{\varepsilon} \partial_{t}\Big( \frac{u_0^2}{2} - u_{\varepsilon} u_{0}\Big) dx = \int_{\T} \rho_{\varepsilon} \big( \partial_{t} + u_{\varepsilon} \partial_{x} \big)\Big(\frac{u_{0}^2}{2} - u_{\varepsilon} u_{0} \Big) dx.
$$
We now apply Lemma \eqref{key_identity} with the function $\psi = \frac{u_{0}^2}{2} - u_{\varepsilon} u_{0}.$ Thus,
$$
\frac{d}{dt} \mathcal{E}_{kin}(t) = \int_{\T} \partial_{t} \Big( \rho_{\varepsilon}\Big( \frac{u_{0}^2}{2} - u_{\varepsilon} u_{0} \Big) \Big) + \partial_{x} \Big( \rho_{\varepsilon} u_{\varepsilon} \Big( \frac{u_{0}^2}{2} -u_{\varepsilon} u_{0} \Big) \Big) dx.
$$
We set for ease $I_{1}(t) = \int_{\T} \partial_{t} \Big( \rho_{\varepsilon} \frac{u_{0}^2}{2} \Big) + \partial_{x}\Big( \rho_{\varepsilon} u_{\varepsilon} \frac{u_{0}^2}{2} \Big) dx,$
and
$
I_{2}(t) = - \int_{\T} \partial_{t}\Big( \rho_{\varepsilon} u_{\varepsilon} u_{0} \Big) + \partial_{x} \Big( \rho_{\varepsilon} u_{\varepsilon}^2 u_{0} \Big) dx.
$
We begin to treat $I_{1}.$
Multiplying the continuity equation of \eqref{EPB} by $\frac{u_{0}^2}{2}$ we obtain
$$
\partial_{t} \Big (\rho_{\varepsilon} \frac{u_{0}^2}{2} \Big) + \partial_{x} \Big( \rho_{\varepsilon} u_{\varepsilon} \frac{u_{0}^2}{2} \Big) = \rho_{\varepsilon} ( \partial_{t} + u_{\varepsilon} \partial_{x} )\Big( \frac{u_{0}^2}{2} \Big).
$$
Therefore,
$$
I_{1}(t) = \int_{\T} \rho_{\varepsilon} ( \partial_{t} + u_{\varepsilon} \partial_{x} )\Big( \frac{u_{0}^2}{2} \Big) dx = \int_{\T} \rho_{\varepsilon} u_{0}( \partial_{t} + u_{\varepsilon} \partial_{x} )( u_{0} ) dx
$$
As for $I_{2}$, we multiply the momentum equation of \eqref{EPB} by $u_{0}$ to get
$$
\partial_{t}\big( \rho_{\varepsilon} u_{\varepsilon} u_{0} \big) + \partial_{x}\big( \rho_{\varepsilon} u_{\varepsilon}^2 u_{0} \big) = 
\rho_{\varepsilon} u_0\partial_{x} \phi_{\varepsilon} 
+ \rho_{\varepsilon} u_{\varepsilon} \big( \partial_{t} + u_{\varepsilon} \partial_{x} \big)(u_{0}).
$$
Therefore,
$$
I_{2}(t) = - \int_{\T}\Big[ \rho_{\varepsilon}u_0 \partial_{x} \phi_{\varepsilon}  + \rho_{\varepsilon} u_{\varepsilon} \big( \partial_{t} + u_{\varepsilon} \partial_{x} \big)(u_{0})\Big] dx.
$$
Combining $I_{1}$ and $I_{2}$ we glean,
\begin{align*}
&\frac{d}{dt} \mathcal{E}_{kin}(t) = \int_{\T} \Big[\rho_{\varepsilon} u_{0} ( \partial_{t} + u_{\varepsilon} \partial_{x} )( u_{0} ) -   \rho_{\varepsilon} u_0\partial_{x} \phi_{\varepsilon}  - \rho_{\varepsilon} u_{\varepsilon} \big( \partial_{t} + u_{\varepsilon} \partial_{x} \big)(u_{0}) \Big]dx  \\
&= \int_{\T} \Big[\rho_{\varepsilon} (u_{0} - u_{\varepsilon}) ( \partial_{t} + u_{\varepsilon} \partial_{x} )( u_{0} )  -   \rho_{\varepsilon}u_0 \partial_{x} \phi_{\varepsilon}\Big]  dx\\
&= \int_{\T} \rho_{\varepsilon} (u_{0}-u_{\varepsilon})( \partial_{t} + u_{0} \partial_{x} )( u_{0})dx - \int_{\T} \rho_{\varepsilon} (u_{0}-u_{\varepsilon})^2 \partial_{x} u_{0} dx
-\int_{\T} \rho_{\varepsilon} \partial_{x} \phi_{\varepsilon} u_{0} dx.
\end{align*}
Using the Poisson equation, we have for the last term 
\begin{align*}
&\int_{\T} \rho_{\varepsilon} \partial_{x} \phi_{\varepsilon} u_{0} dx = \int_{\T} \big( \varepsilon^2 \partial_{xx} \phi_{\varepsilon} + e^{-\phi_{\varepsilon}} \big) \partial_{x} \phi_{\varepsilon} u_{0} dx = \int_{\T}  \partial_{x}\Big(  \frac{\varepsilon^2}{2} \vert  \partial_{x} \phi_{\varepsilon} \vert ^2 - e^{-\phi_{\varepsilon}} \Big) u_{0}dx\\
&= -\int_{\T} \Big(\frac{\varepsilon^2}{2} \vert  \partial_{x} \phi_{\varepsilon} \vert ^2 - e^{-\phi_{\varepsilon}} \Big) \partial_{x} u_{0} dx.
\end{align*}
Therefore,
\begin{align}
\label{dt-Ekin}
&\frac{d}{dt} \mathcal{E}_{kin}(t) =  \int_{\T} \Big(\frac{\varepsilon^2}{2} \vert  \partial_{x} \phi_{\varepsilon} \vert ^2 - e^{-\phi_{\varepsilon}} \Big) \partial_{x} u_{0} dx - \int_{\T} \rho_{\varepsilon} (u_{0}-u_{\varepsilon})^2 \partial_{x} u_{0} dx \\
&+\int_{\T} \rho_{\varepsilon} (u_{0}-u_{\varepsilon})( \partial_{t} + u_{0} \partial_{x} )( u_{0} )dx.\nonumber
\end{align}
As for $\mathcal{E}_{int}$, we have 
\begin{eqnarray}
-\frac{d}{dt} \mathcal{E}_{int}(t) &=&  \frac{d}{dt} \int_{\T} (\rho_0 - e^{-\phi_\varepsilon} \log\rho_0) dx \nonumber\\
&=& \int_{\T} \partial_t \rho_0 dx  - \int_{\T} \Big[(\partial_t e^{-\phi_\varepsilon}) \log\rho_0 + e^{-\phi_\varepsilon}\frac{\partial_t \rho_0}{\rho_0}\Big] dx\nonumber\\
&=& \int_{\T} \Big(-\frac{e^{-\phi_\varepsilon}}{\rho_0} + 1\Big)\partial_t \rho_0 dx + \int_{\T}(\varepsilon^2 \partial_{txx} \phi_\varepsilon-\partial_t \rho_\varepsilon) \log \rho_0 dx\nonumber\\
\label{dtEinte}
&=& \int_{\T} \Big(-\frac{e^{-\phi_\varepsilon}}{\rho_0} + 1\Big)\partial_t \rho_0 dx + \int_{\T}\varepsilon^2 \partial_{txx} \phi_\varepsilon \log \rho_0 dx +\int_{\T} \partial_x (\rho_{\varepsilon}u_{\varepsilon}) \log \rho_0 dx,\nonumber \\
&=& \int_{\T} \frac{e^{-\phi_\varepsilon}}{\rho_0}(\rho_0\partial_x u_0 + u_0\partial_x \rho_0)dx + \int_{\T}\varepsilon^2 \partial_{txx} \phi_\varepsilon \log \rho_0 dx +\int_{\T} \partial_x (\rho_{\varepsilon}u_{\varepsilon}) \log \rho_0 dx,\nonumber\\
\label{dtEinte2}
&=& \int_{\T} e^{-\phi_\varepsilon}\partial_x u_0 dx+\int_{\T} e^{-\phi_\varepsilon} u_0\partial_x \log \rho_0 dx + \int_{\T}\varepsilon^2 \partial_{txx} \phi_\varepsilon \log \rho_0 dx +\int_{\T} \partial_x (\rho_{\varepsilon}u_{\varepsilon}) \log \rho_0 dx,
\end{eqnarray}
where we have used that $\int_{\T} \partial_{t} \rho_{0} dx = 0$. 

We then observe that 
\begin{eqnarray*}
\int_\T \rho_\varepsilon (u_0-u_\varepsilon) \partial_x \log \rho_0 dx &=& \int_\T \rho_\varepsilon u_0 \partial_x \log \rho_0 dx - \int_\T (\rho_{\varepsilon}u_{\varepsilon}) \partial_x \log \rho_0 dx\nonumber\\
&=& \int_{\T} e^{-\phi_\varepsilon} u_0 \partial_x \log \rho_0 dx +\varepsilon^2\int_{\T} \partial_{xx}\phi_\varepsilon u_0 \partial_x \log \rho_0 dx + \int_\T \partial_x (\rho_{\varepsilon}u_{\varepsilon}) \log \rho_0 dx, 
\end{eqnarray*}
so that \eqref{dtEinte2} rewrites 
\begin{eqnarray}
-\frac{d}{dt} \mathcal{E}_{int}(t)&=&\int_{\T} e^{-\phi_\varepsilon}\partial_x u_0 dx +\int_\T \rho_\varepsilon (u_0-u_\varepsilon) \partial_x \log \rho_0 dx
\nonumber\\
\label{dt-Eint}
&&+ \int_{\T}\varepsilon^2 \partial_{txx} \phi_\varepsilon \log \rho_0 dx-\varepsilon^2\int_{\T} \partial_{xx}\phi_\varepsilon u_0 \partial_x \log \rho_0 dx. 
\end{eqnarray}
Gathering the equalities \eqref{dt-Ekin} and \eqref{dt-Eint}, we get
\begin{align}
&\frac{d}{dt}\mathcal{E}(t) = \int_{\T} \Big(\frac{\varepsilon^2}{2} \vert  \partial_{x} \phi_{\varepsilon} \vert ^2 - e^{-\phi_{\varepsilon}} \Big) \partial_{x} u_{0} dx - \int_{\T} \rho_{\varepsilon} (u_{0}-u_{\varepsilon})^2 \partial_{x} u_{0} dx\\
&+\int_{\T} \rho_{\varepsilon} (u_{0}-u_{\varepsilon})( \partial_{t} + u_{0} \partial_{x} )( u_{0})dx \\
&+\int_{\T} e^{-\phi_\varepsilon}\partial_x u_0 dx +\int_\T \rho_\varepsilon (u_0-u_\varepsilon) \partial_x \log \rho_0 dx\\
& +\int_{\T}\varepsilon^2 \partial_{txx} \phi_\varepsilon \log \rho_0 dx-\varepsilon^2\int_{\T} \partial_{xx}\phi_\varepsilon u_0 \partial_x \log \rho_0 dx\\
&=\int_\T \rho_\varepsilon (u_0-u_\varepsilon) \Big[\partial_t u_0+u_0\partial_x u_0 + \partial_x\log \rho_0 \Big]dx - \int_{\T} \rho_{\varepsilon} (u_{0}-u_{\varepsilon})^2 \partial_{x} u_{0} dx \\
\label{calcul1}
&+ \int_{\T} \frac{\varepsilon^2}{2} \vert  \partial_{x} \phi_{\varepsilon} \vert ^2\partial_x u_0 dx
+\int_{\T}\varepsilon^2 \partial_{txx} \phi_\varepsilon \log \rho_0 dx-\varepsilon^2\int_{\T} \partial_{xx}\phi_\varepsilon u_0 \partial_x \log \rho_0 dx.  
\end{align} 
Moreover, $u_0$ satisfies  
$\partial_t u_0 + u_{0} \partial_{x}u_{0} +\partial_x\log\rho_0
=0$ so that we eventually obtain 
\begin{align*}
&    \frac{d}{dt} \mathcal{E}(t) =  \int_{\T} \Big[ -\rho_{\varepsilon}(u_{0}-u_{\varepsilon})^2 + \frac{\varepsilon^2}{2} \vert \partial_{x} \phi_{\varepsilon} \vert^2 
\Big]\partial_{x} u_{0} dx +\varepsilon^2\int_{\T}\Big[ \partial_{txx} \phi_\varepsilon \log \rho_0 - \partial_{xx}\phi_\varepsilon u_0 \partial_x \log \rho_0 \Big]dx.
\end{align*}
Hence, we deduce the following inequality  
$$
\frac{d}{dt}\mathcal{E}(t) \leq \varepsilon^2 {\cal L}(t) + 2\mathcal{E}(t)\|\partial_x u_0\|_{L_{t,x}^{\infty}}, 
$$
with ${\cal L}(t) =\int_{\T} [\partial_{txx} \phi_\varepsilon \log \rho_0 - \partial_{xx}\phi_\varepsilon u_0 \partial_x \log \rho_0 ]dx$ and a Grönwall lemma enables us to conclude.\\
\end{proof}
Note that in \cite{Han-Kwan01082011}, it is explained how to bound $\mathcal{L}(t)$ with respect to $\varepsilon.$

\section{Discretization}

Let $\Delta x = \frac{1}{N+1}$ where $N \in \N^{\star}$ is fixed. We consider a uniform grid defined by the sequence of points $(x_{i} = i \Delta x)_{i \in \Z}$. Since we work on $\T$, we shall identify two points of the same grid according to the equivalence relation defined on $\R$ by
\begin{equation} \label{equiv_relation_1}
\forall x, y \in \R, \quad x \equiv y \: \textnormal{mod} \: \Z \Leftrightarrow x-y \in \Z.
\end{equation}
It yields in particular an identification of the torus $\T$ with the unit interval $[0,1)$ through the unique  decomposition of a real number:
\begin{equation}
    \forall x \in \R, \: x = \lfloor x \rfloor + \lbrace x \rbrace
\end{equation}
where $\lfloor \cdot \rfloor$ denotes the integer part function and $\lbrace  \cdot \rbrace$ denotes the fractional part function. Especially,
\begin{equation} \label{identification_torus}
\forall x \in \R, \: x \equiv \lbrace x \rbrace \:  \textnormal{mod} \: \Z.
\end{equation}
The relation \eqref{identification_torus} applied to the grid points $(i \Delta x)_{i \in \Z}$ yields 
\begin{equation}
    \forall i \in \Z,  \: \exists ! i^{\star} \in \lbrace 0,\dots,N \rbrace \quad x_{i} \equiv x_{i^{\star}} \: \textnormal{mod} \:  \Z,
\end{equation}
where $i^{*}$ is the remainder of the Euclidean division of $i$ by $(N+1).$
We then idenfity the quotient space $\Z / (N+1) \Z$ with the first $N+1$ non negative integers. So, in the following we shall systematically identify an integer with its unique representation in $\lbrace 0,\dots,N \rbrace.$ To approximate the solutions of \eqref{EPB}, we use a finite volume approach where $\T$ is covered by a union of non empty disjoints intervals of size $\Delta x$.  
We then define two meshes
\begin{equation}
\mathcal{T} := \bigcup_{i = 0}^{N} C_{i}, \quad C_{i} = \Big [x_{i}- \frac{\Delta x}{2} , x_{i}+ \frac{\Delta x}{2} \Big), \quad \mathcal{T}^{\star} := \bigcup_{i = 0}^{N} C_{i+\frac{1}{2}}, \quad C_{i + \frac{1}{2}} = \Big [x_{i+\frac{1}{2}}- \frac{\Delta x}{2} , x_{i+\frac{1}{2}}+ \frac{\Delta x}{2} \Big),
\end{equation}
where $x_{i+\frac{1}{2}} = x_{i} + \frac{\Delta x}{2}$, $\mathcal{T}$ is called the primal mesh and $\mathcal{T}^{\star}$ is called the dual mesh. We then consider two spaces of piecewise constant functions on $\T$: 
\begin{align}
    X(\mathcal{T}) = \Big \lbrace v \in L^{1}_{\textnormal{loc}}(\T;\R), \:  \forall i \in \lbrace 0,\dots,N \rbrace, \: v(x) = v_{i} :=\frac{1}{\Delta x} \int_{C_{i}} v(x) dx \textnormal{ if } x \in C_{i} \Big \rbrace, \\
    X(\mathcal{T}^{\star}) =\Big \lbrace v \in L^{1}_{\textnormal{loc}}(\T;\R), \:  \forall i \in \lbrace 0,\dots,N \rbrace, \: v(x) = v_{i+\frac{1}{2}}:=\frac{1}{\Delta x} \int_{C_{i+\frac{1}{2}}} v(x) dx \textnormal{ if } x \in C_{i+\frac{1}{2}} \Big \rbrace.
\end{align}
A natural space to estimate the solution to the Poisson equation is $H^{1}(\T)$ endowed with its canonical norm. We consider the discrete analogue for functions in $X(\mathcal{T})$ by introducing the discrete gradient defined by
\begin{equation}\label{discrete_gradient}
    \forall \varphi \in X(\mathcal{T}), \quad (\delta \varphi) \in X(\mathcal{T}^{\star}) \textnormal{ and } (\delta \varphi)_{i+\frac{1}{2}} =\frac{ \varphi_{i+1}-\varphi_{i} }{\Delta x}, i\in \lbrace 0,\dots,N\rbrace.
\end{equation}
The discrete $H^{1}$ semi-norm is defined by
\begin{align}
    \label{def-h1-semi-norm}
    \forall \varphi \in X(\mathcal{T}), \quad \big | \varphi \big |_{H^{1}(\T)} = \Bigg( \sum_{i=0}^{N} \Big \vert \frac{\varphi_{i+1}-\varphi_{i}}{\Delta x} \Big \vert^2 \Delta x \Bigg)^\frac{1}{2}.
\end{align}
The discrete Laplacian is defined for functions in $X(\mathcal{T})$ by 

\begin{equation}\label{discrete_laplacian}
    \forall \varphi \in X(\mathcal{T}), \quad (\Delta \varphi) \in X(\mathcal{T}) \textnormal{ and } (\Delta \varphi)_{i} =\frac{ \varphi_{i+1}-2\varphi_{i}  + \varphi_{i-1}}{\Delta x^2}, i\in \lbrace 0,\dots,N\rbrace.
\end{equation}
We will use routinely a discrete analogue of the integration by parts formulas. More precisely, we have.
\begin{lemma}{(Discrete integration by parts)} \label{ipp_lemma}
It holds,
\begin{align}
\forall (v,p) \in X(\mathcal{T}^{\star}) \times X(\mathcal{T}), \quad \sum_{i=0}^N v_{i+1/2}(p_{i+1} - p_i) \Delta x = -\sum_{i=0}^N (v_{i+1/2}-v_{i-1/2})p_i \Delta x, \label{ipp1}\\
\forall (\varphi,\psi) \in X(\mathcal{T})^2, \quad \sum_{i=0}^{N} (\Delta \varphi)_{i} \psi_{i} \Delta x = - \sum_{i=0}^{N} (\delta \phi)_{i+\frac{1}{2}} (\delta \psi)_{i+\frac{1}{2}} \Delta x. \label{ipp2}
\end{align}
\end{lemma}
\begin{proof}
Both \eqref{ipp1} and \eqref{ipp2} are obtained using a translation of indices and using the periodicity. We only prove \eqref{ipp1}. Using a translation of indices, we get
\begin{eqnarray}
\sum_{i=0}^N v_{i+1/2}(p_{i+1} - p_i) \Delta x&=& \sum_{i=0}^N v_{i+1/2}p_{i+1} \Delta x- \sum_{i=0}^N v_{i+1/2}p_i \Delta x =  \sum_{i=1}^{N+1} v_{i-1/2}p_{i} \Delta x- \sum_{i=0}^N v_{i+1/2}p_i  \Delta x \nonumber\\
&=& -\sum_{i=0}^{N} (v_{i+1/2} - v_{i-1/2})p_{i}\Delta x - v_{-1/2}p_0+v_{N+1/2}p_{N+1}. \nonumber
\end{eqnarray}
By periodicity, we have $v_{-1/2}=v_{N+1/2}$ 
and $p_0=p_{N+1}$ which gives the result. 
\end{proof}
As we are mainly concerned with discrete analogues of \eqref{nrj_cons} and \eqref{modulated_energy_estimate}, we introduce a discrete analogue of the energy functional and of its modulated version. We thus define  the discrete total energy for $(\r,u,\phi) \in X(\mathcal{T}) \times X(\mathcal{T}^{\star}) \times X(\mathcal{T})$ by
\begin{equation} \label{def:H_discret}
   \mathcal{H}(\rho,u,\phi) = \sum_{i=0}^{N} \Big(  \rho_{i+\frac{1}{2}}\frac{u_{i+\frac{1}{2}}^2}{2}+ \frac{\varepsilon^2}{2} \Big \vert \frac{  \phi_{i+1}-\phi_{i}}{\Delta x } \Big \vert^2 + h(\phi_{i}) \Big) \Delta x.
\end{equation}
The discrete modulated version around a constant state  $(\bar{u},\bar{\phi}) \in X(\mathcal{T}^{\star})\times X(\mathcal{T})$ is given by 
\begin{align}
\label{modulated_energy_discrete}
 \quad {\cal E}(\rho,u,\phi \:  | \:  \bar{u},\bar{\phi})  &= \sum_{i=0}^{N} \rho_{i+\frac{1}{2}} \frac{(u_{i+\frac{1}{2}}- \bar{u})^2}{2} \Delta x + \sum_{i=0}^{N} \frac{\varepsilon^2}{2} \Big \vert \frac{ \phi_{i+1}-\phi_{i}}{\Delta x} \Big \vert^2 \Delta x \\
&+ \sum_{i=0}^{N}\Big[ \tilde{h}(e^{-\phi_{i}}) - \Big(\tilde{h}(e^{-\bar{\phi}})+\tilde{h}'(e^{-\bar{\phi}})(e^{-\phi_{i}}-e^{-\bar{\phi}})  \Big) \Big]\Delta x. \nonumber
\end{align}
As in \eqref{decomp_E}, expanding the first term we have the decomposition
\begin{equation} \label{discrete_decomp_E}
    {\cal E}(\rho,u,\phi|\bar{u}, \bar{\phi}) = \mathcal{H}(\rho,u,\phi) + {\cal E}_{kin}(\rho, u|\bar{u}) - {\cal E}_{int}(\phi|\bar{\phi})
\end{equation}
where $\mathcal{H}$ is given in \eqref{def:H_discret} and
\begin{align}
    &{\cal E}_{kin}(\rho, u|\bar{u}) = \sum_{i=0}^{N} \rho_{i+\frac{1}{2}} \Big(  \frac{\bar{u}^2 }{2} - u_{i+\frac{1}{2}}\bar{u} \Big) \Delta x,\\
   & {\cal E}_{int}(\phi | \bar{\phi}) = \sum_{i=0}^{N} \Big( \tilde{h}(e^{-\bar{\phi}})+\tilde{h}'(e^{-\bar{\phi}})(e^{-\phi_{i}}-e^{-\bar{\phi}})  \Big) \Delta x.
\end{align}

\subsection{The time implicit scheme}\label{sec:discrete_scheme}
We fix $N_{T} \in \N^{\star}$ and set  $\Delta t = \frac{T}{N_{T}}$. We consider a uniform in time discretization $(t_{n})_{n = 0,\dots,N_{T}} = (n \Delta t )_{n= 0,\dots,N_{T}}$.
For each $n= 0,\dots,N_{T}$, we consider
$(\rho^{n}, u^{n}, \phi^{n}) \in X(\mathcal{T})\times X(\mathcal{T}^{*}) \times X(\mathcal{T})$ defined by induction 
for $n \in \lbrace 0,\dots,N_{T}-1 \rbrace$ :
\begin{equation} \label{EPBD}
\begin{cases}
&\displaystyle  \frac{ \rho_i^{n+1} - \rho_i^n}{\Delta t} + \frac{1}{\Delta x}({\cal F}_{i+1/2}^{n+1}-{\cal F}_{i-1/2}^{n+1})= 0,  \\
&\displaystyle  \frac{ \rho_{i+\frac{1}{2}}^{n+1} u_{i+\frac{1}{2}}^{n+1} - \rho_{i+\frac{1}{2}}^{n} u_{i+\frac{1}{2}}^{n} }{\Delta t} + \frac{1}{\Delta x}({\cal F}_{i+1}^{n+1}u_{i+1}^{n+1}-{\cal F}_{i}^{n+1}u_{i}^{n+1}) =  \tilde{\rho}_{i+\frac{1}{2}}^{n+1} \delta (\phi^{n+1})_{i+\frac{1}{2}}, \\
&\displaystyle  \varepsilon^2 \Delta(\phi^{n+1})_{i}  + e^{-\phi_i^{n+1}}  = \rho^{n+1}_i,
\end{cases}
\end{equation}
where $i \in \lbrace 0,\dots,N \rbrace$. The system \eqref{EPBD} is supplemented with an initial condition $\rho^{0} \in X(\mathcal{T})$ and $u^{0} \in X(\mathcal{T}^{\star})$ such that
\begin{align}\label{init_cond_discrete}
\rho^{0}_{i} = \frac{1}{\Delta x} \int_{C_{i}}  \rho_{\varepsilon}^{\textnormal{ini}}(x) dx, \quad u^{0}_{i+\frac{1}{2}} = \frac{1}{\Delta x} \int_{C_{i+\frac{1}{2}}} u_{\varepsilon}^{\textnormal{ini}}(x) dx.
\end{align}
For the sake of conciseness in the notation we have voluntarily discarded the dependence with respect to $\varepsilon$ of the discrete solution to \eqref{EPBD}. In \eqref{EPBD}, the flux of mass at the interface $x_{i+\frac{1}{2}}$ is defined by
\begin{equation} \label{flux_of_mass}
    \mathcal{F}_{i+\frac{1}{2}}^{n+1} = G(\rho_{i}^{n+1},\rho_{i+1}^{n+1},u_{i+\frac{1}{2}}^{n+1})
\end{equation}
where $G : \R^{3} \longrightarrow \R$ is the function defined by
\begin{equation} \label{def:G}
  \forall (s,t,u) \in \R^3, \:   G(s,t,u) = s g(u) - t (g(u)-u)
\end{equation}
where $g : \R \rightarrow \R$ is some arbitrary function that satisfies the following assumptions:
\begin{align}
    g \in \textnormal{Lip}(\R), \label{ass_1_g}\\
    \forall u \in \R, \quad g(u) \geq \max(u,0), \label{ass_2_g}\\
    g \textnormal{ is differentiable at } \: 0. \label{ass_3_g}
\end{align}
Under the assumptions \eqref{ass_1_g}, \eqref{ass_2_g} and \eqref{ass_3_g}, the function $G$ is a Lipschitz continuous, non decreasing in its first variable and non increasing in its second variable. Moreover, it verifies the usual consistency relation in the finite volume sense:
\begin{align}
    \forall (\rho,u) \in \R^2, \: G(\rho,\rho,u) = \rho u.\label{consistency_relation}
\end{align}
The function $g$ could be thought as some regularization of the function $u \mapsto \max(u,0).$ The density and its flux of mass at the interface $x_{i+\frac{1}{2}}$ are defined by:
\begin{equation}
    \label{rho_dual}
    \rho_{i+\frac{1}{2}}^{n+1} = \frac{ \rho_{i+1}^{n+1} + \rho_{i}^{n+1} }{2}, \quad
    \mathcal{F}_{i}^{n+1} = \frac{ \mathcal{F}_{i+\frac{1}{2}}^{n+1} + \mathcal{F}_{i-\frac{1}{2}}^{n+1}}{2}. 
\end{equation}
An elementary consequence of these two definitions and of the discrete continuity equation is:
\begin{align*}
&\exists i \in \lbrace 0,\dots,N \rbrace, \: \forall j \in \lbrace i,i+1\rbrace :\;   \frac{ \rho_j^{n+1} - \rho_j^n}{\Delta t} + \frac{1}{\Delta x}({\cal F}_{j+1/2}^{n+1}-{\cal F}_{j-1/2}^{n+1})= 0 \\
&\Longrightarrow     \frac{ \rho_{i+\frac{1}{2}}^{n+1} - \rho_{i+\frac{1}{2}}^n}{\Delta t} + \frac{1}{\Delta x}({\cal F}_{i+1}^{n+1}-{\cal F}_{i-1}^{n+1})= 0.
\end{align*}
It embodies the fact that if the discrete continuity equation is satisfied on the cells $C_{i}$ and $C_{i+1}$ for some $i$ then it has a dual analogue on the cell $C_{i+\frac{1}{2}}$ which is in between.
The velocity at the interface $x_{i}$ is defined by:
\begin{equation} \label{u_primal}
    u_{i}^{n+1} = \begin{cases}
        u_{i-\frac{1}{2}}^{n+1} \textnormal{ if } \mathcal{F}_{i}^{n+1} \geq 0,\\
        u_{i+\frac{1}{2}}^{n+1} \textnormal{ if }\mathcal{F}_{i}^{n+1} < 0.
    \end{cases}
\end{equation}
An originality of this work is a definition of the density in the forcing term of discrete momentum equation given by:
\begin{equation} \label{rho_upwind}
    \tilde{\rho}_{i+\frac{1}{2}}^{n+1} = \begin{cases}
\displaystyle \frac{ G(\rho_{i}^{n+1},\rho_{i+1}^{n+1},u_{i+\frac{1}{2}}^{n+1}) - G(\rho_{i}^{n+1},\rho_{i+1}^{n+1},0) }{ u_{i+\frac{1}{2}}^{n+1} - 0 }, \;\;\;     \textnormal{ if } u_{i+\frac{1}{2}}^{n+1} \neq 0,\\
             \\
        \rho_{i+1}^{n+1} -(\rho_{i+1}^{n+1}-\rho_{i}^{n+1}) g'(0), \;\;\;\hspace{2.1cm} \textnormal{ if } u_{i+\frac{1}{2}}^{n+1} = 0,
    \end{cases}
\end{equation} 
which yields the unconditional dissipativity of the force term in the sense that
\begin{align} \label{dissipative_force}
  &\forall n \in \lbrace 0,\dots,N_{T}-1 \rbrace, \quad   \sum_{i=0}^{N} \tilde{\rho}_{i+\frac{1}{2}}^{n+1} u_{i+\frac{1}{2}}^{n+1} \delta (\phi^{n+1})_{i+\frac{1}{2}} \Delta x + \frac{\varepsilon^2}{2} | \phi^{n+1} |_{H^{1}(\T)}^2 +  \sum_{i=0}^{N} h(\phi_{i}^{n+1}) \Delta x \\
  &\leq  \frac{\varepsilon^2}{2} | \phi^{n} |_{H^{1}(\T)}^2+  \sum_{i=0}^{N} h(\phi_{i}^{n}) \Delta x. \nonumber
\end{align}

Note that the function $g$ is differentiable at the origin, so the definition \eqref{rho_upwind} makes $\tilde{\rho}_{i+\frac{1}{2}}^{n+1}$ a continuous functions of each of its arguments. Indeed we prove.
\begin{lemma}{(Continuity of $\tilde{\rho}$)} The function $\tilde{\rho} : \R^{3} \longrightarrow \R$ given by
\begin{align}
    \tilde{\rho}(s,t,u) = \begin{cases}
        \frac{G(s,t,u) - G(s,t,0)}{u-0} \quad (s,t,u) \in \R^{3} \setminus F ,\\
        t - (t-s) g'(0) \quad (s,t,u) \in F,
    \end{cases}
\end{align}
where $F = \Big\lbrace (s,t,0) \: :  (s,t) \in \R^2 \Big\rbrace$
is continuous on $\R^{3}.$
\begin{proof} Since $g$ is a Lipschitz continuous on $\R$, the function $G$ is continuous on $\R^{3}$ as a sum and product of continuous functions. Thus,  $\tilde{\rho}$ is a continuous function in the open set $\R^{3} \setminus F$. Then remark that for $(s,t,u) \in \R^{3} \setminus F$ we have 
\begin{align}
    \tilde {\rho}(s,t,u) = t -(t-s) \hat{g}(u)
\end{align}
where  $\hat{g}(u) = \begin{cases}
   \displaystyle \frac{g(u) - g(0)}{u} \; \mbox{ if } \; u \neq 0, \\
    g'(0) \quad\quad\quad \,\, \mbox{ if } \; u = 0.
\end{cases}$
Observe that $\hat{g}$ is a continuous function on $\R$ since $g$ is continuous on $\R^{\star}$ and it is differentiable at the origin. For $(s,t,u) \in \R^{3} \setminus F$ we have
\begin{align}
    \Big \vert \tilde{\rho}(s,t,u) - \big(t - (t-s)) g'(0)\big) \Big \vert  = \Big \vert t-s \Big \vert \Big \vert \hat{g}(u) - g'(0) \Big \vert.
\end{align}
By continuity of $\hat{g}$ we deduce that $\tilde{\rho}(s,t,u) \longrightarrow \tilde{\rho}(s,t,0)$ as $(s,t,u) \longrightarrow (s,t,0).$ 
\end{proof}
\end{lemma}

The previous lemma is important in view of the existence theory for \eqref{EPBD} since we shall invoke the Brouwer fixed-point theorem which requires the continuity of an appropriate map. We mention that the main idea behind \eqref{rho_upwind} is the need of  compatibility between the discrete continuity equation and the Poisson equation  to get discrete energy estimates. More precisely, our definition \eqref{rho_upwind} enables us to reproduce a discrete version of the following computation:
\begin{equation*}
    \int_{\T} \rho_{\varepsilon} u_{\varepsilon}\partial_{x} \phi_{\varepsilon} dx= - \int_{\T} \partial_{x}(\rho_{\varepsilon} u_{\varepsilon} ) \phi_{\varepsilon} dx = \int_{\T} \partial_{t} \rho_{\varepsilon} \phi_{\varepsilon} dx.
\end{equation*}
Then, using the discrete Poisson equation, the implicitness in time yields some expected dissipation and it turns out that the additional spatial part  also brings some dissipation since $g(0) \geq 0$. We also mention that the forcing term in the momentum equation is a priori not written as a gradient in space on the contrary to what is proposed in \cite{degond_b}. We nevertheless highlight the fact for the pressure-less Euler-Poisson equation, the theory of weak solutions is not fully understood. In our setting, we always consider strong solutions. Besides note that for $\varepsilon > 0$ fixed, for each $t \in [0,T]$, $\phi_{\varepsilon}(t)$ gains two spatial derivatives thanks to the standard elliptic regularity theory for the Poisson equation. Thus, the product $\rho_{\varepsilon}(t) \partial_{x} \phi_{\varepsilon}(t)$ is always a function even with $\rho_{\varepsilon}(t) \in L^{1}(\T).$ Of course, what is more challenging is the case $\varepsilon \rightarrow 0$ since the Poisson equation becomes algebraic at the limit and there is no gain of regularity for $\phi_{0}.$  The study of the convergence as $\varepsilon \rightarrow 0$ towards a weak entropic solution to \eqref{EPB0} is to the best of our knowledge an open question. We may anyway, at the discrete level, always consider consistency of the scheme \eqref{EPBD} for strong solutions.  By the way, our modulated energy estimates in the limit $\varepsilon \rightarrow 0$ holds only for constant solutions.
Last but not the least, note that the definition \eqref{rho_upwind} does not yield straightforwardly $\tilde{\rho}^{n+1}_{i+\frac{1}{2}} \geq 0$ if $\rho^{n+1} > 0$ on $\T.$ Nevertheless, we do observe that it can be re-written for $i \in \lbrace 0,..,N \rbrace$ as
\begin{align} \label{rho_upwind_2}
\tilde{\rho}^{n+1}_{i+\frac{1}{2}} = \rho_{i+1}^{n+1}(1-\hat{g}(u_{i+\frac{1}{2}}^{n+1})) + \rho_{i}^{n+1}\hat{g}(u_{i+\frac{1}{2}}^{n+1})
\end{align}
where $\hat{g}(u) = \begin{cases}
   \displaystyle \frac{g(u) - g(0)}{u} \; \mbox{ if } \; u \neq 0, \\
    g'(0) \quad\quad\quad \,\, \mbox{ if } \; u = 0.
\end{cases}
$ 
As a matter of fact, if $g$ is non decreasing and its Lipschitz constant is such that $\textnormal{Lip}(g) \leq 1$ then $\tilde{\rho}_{i+\frac{1}{2}}^{n+1}$ given by \eqref{rho_upwind_2} is a convex combination of $\rho_{i+1}^{n+1}$ and $\rho_{i}^{n+1}$ so it preserves the $L^{\infty}$ boundedness provided $\rho^{n+1}$ is. Since this assumption is not needed in the discrete stability estimates, we shall not consider it as granted.

The first main result of this work is.
\begin{theorem}{(Unconditional Energy stability)} \label{main_result_1}
Let $N_{T}$ and $N \in \N$ be two positive integers and $(\rho^{0},u^{0}) \in X(\mathcal{T}) \times X(\mathcal{T}^{\star})$ being given in \eqref{init_cond_discrete} such that $\rho^{0} > 0$. Assume $\phi^{0} \in X(\mathcal{T})$ verifies the discrete Poisson equation initially. Then, there exists a solution $(\rho^{n},u^{n},\phi^{n})_{n = 0,\dots,N_{T}} \subset X(\mathcal{T}) \times X(\mathcal{T}^{\star}) \times X(\mathcal{T})$ to the scheme \eqref{EPBD}. In addition, this solution satisfies for  $0 \leq n \leq N_{T}$ the estimates:
\begin{align}
    \rho^{n} > 0 \textnormal{ on } \T, \label{positivity}\\
 \| \rho^{n} \|_{L^{1}(\T)} = \| \rho^{0} \|_{L^{1}(\T)}, \label{conservativity}\\
 \mathcal{H}(\rho^{n},u^{n},\phi^{n}) + \Delta t\sum_{k=1}^{n-1} \tau^{n} = \mathcal{H}(\rho^{0},u^{0},\phi^{0}) \label{energy_estimates}
\end{align}
where $\tau^{n} \geq 0$ is given in \eqref{tau_n} and $\phi^{n}$ satisfies the elliptic estimates \eqref{estimate_exp_lp}-\eqref{estimate_semi_h1} for $n \in \lbrace 0,\dots,N_{T} \rbrace.$
\end{theorem}
Theorem \ref{main_result_1} states an unconditional energy decay for the scheme \eqref{EPBD}. The second main result of this work is.

\begin{theorem}{(Modulated energy estimates around  constant states with a small velocity)} \label{main_result_2}Let $N_{T}$ and $N$ be two positive integers and let $(\rho^{n},u^{n},\phi^{n})_{n = 0,\dots, N_{T}} \subset X(\mathcal{T}) \times X(\mathcal{T}^{\star}) \times X(\mathcal{T}) $ a solution to \eqref{EPBD} such that $\rho^{n} > 0$ for all $n \in \lbrace 0,\dots,N_{T} \rbrace.$ Assume $\phi^{0} \in X(\mathcal{T})$ verifies the discrete Poisson equation initially. Let $(\bar{u},\bar{\phi}) \in X(\mathcal{T}^{\star}) \times X(\mathcal{T})$ a constant state. Then the solution satisfies for $0 \leq n \leq N_{T} $ the following modulated energy estimates (recalling the definition \eqref{def:H_discret} of ${\cal E}(\rho,u,\phi | \bar{u},\bar{\phi})$):
\begin{itemize}
    \item[a)] If $\bar{u} = 0$ then
    \begin{align} \label{dis_mod_energy_estimate}
        \forall n \in \lbrace 0,\dots,N_{T} \rbrace, \quad {\cal E}(\rho^{n},u^{n},\phi^{n} | 0,\bar{\phi}) \leq {\cal E}(\rho^{0},u^{0},\phi^{0} | 0, \bar{\phi}).
    \end{align}
    \item[b)] If $\bar{u} \neq 0$ and $| \bar{u} | \leq \frac{g(0)}{\textnormal{Lip}(g)}$, provided there exists a constant $K \geq 0$ which is such that $K \underset{ \varepsilon \rightarrow 0}= \mathcal{O}(1),\:  K \underset{ \Delta x \rightarrow 0}= \mathcal{O}(1)$ and $\rho^{n}$ verifies:
    \begin{align}
        \forall n \in \lbrace 0,\dots,N_{T} \rbrace, \quad \| \rho^{n} \|_{L^{\infty}(\T)} \leq K \label{ass_rho_1},
    \end{align}
    then  if $\Delta x \underset{ \varepsilon \rightarrow 0} = \mathcal{O}(\varepsilon)$ we have under the following CFL condition,
    \begin{equation}\label{CFL}
    | \bar{u} | \frac{\Delta t}{\Delta x} \Big( 8  \textnormal{Lip}(g) + 4  + \frac{2 \Delta x^2}{\varepsilon^2} K\Big) < 1
    \end{equation}
    that 
    \begin{equation}
        \forall n \in \lbrace 0,\dots,N_{T} \rbrace, \quad {\cal E}(\rho^{n},u^{n},\phi^{n} | \bar{u},\bar{\phi}) \leq a^{n} {\cal E}(\rho^{0},u^{0},\phi^{0} | \bar{u},\bar{\phi})
    \end{equation}
    with
    \begin{equation}
        a = \frac{1}{ 1 - | \bar{u} | \frac{\Delta t}{\Delta x} \Big( 8  \textnormal{Lip}(g) + 4  + \frac{2 \Delta x^2}{\varepsilon^2} K\Big)}.
    \end{equation}
\end{itemize}
\end{theorem}
Theorem \ref{main_result_2} states an unconditional stability estimates around the constant state $(0, \bar{\phi}).$ This result goes beyond the linearized studies \cite{degond_analysis,degond_b,fabre} around constant states. As for the stability around an arbitrary constant state $(\bar{u},\bar{\phi})$, we have not been able to address the problem in its full generality. It may look paradoxical regarding the case of a null velocity since the system \eqref{EPB} is Galilean invariant.  Nevertheless, the discretization \eqref{EPBD} is a priori not Galilean invariant, except is some special cases when the scheme is equivalent to a Lagragian discretization of the continuity equation. The scheme \eqref{EPBD} is fundamentally Eulerian. We are able to deal with the case where $\bar{u}$ is small enough. In such a case we require the density to be uniformly bounded with respect to $\varepsilon$ and to $\Delta x$. This assumption was already used in the literature \cite{herbin,arun}. Provided the mesh size is of order $\varepsilon$ and a hyperbolic type CFL condition is verified, we are able to prove at most exponential growth of the modulated energy with a rate of increase which is bounded uniformly in $\varepsilon.$ This is to compare with its continuous analogue \eqref{modulated_energy_estimate} which is different in nature. The difficulty comes from the fact that the force term is not conservative at the discrete level:
\begin{equation} \label{non_cons_force}
  \forall n \in \lbrace 0,\dots,N_{T}-1 \rbrace, \quad   \sum_{i=0}^{N} \tilde{\rho}_{i+\frac{1}{2}}^{n+1}\delta (\phi^{n+1})_{i+\frac{1}{2}} \Delta x = 0 + \textnormal{non zero residual terms}.
\end{equation}
The residual terms need to be controlled in the worst case by the modulated energy at step $n+1$, this is precisely where the smallness condition on $\bar{u}$ and the CFL condition comes from. We mention that the assumption that $\Delta x$ is of order $\varepsilon$  seems us not so natural regarding what is expected for the spatial regularity of the perturbation.  Indeed, in \cite{puguo}, it is proven that for strong solutions, we have  $\frac{1}{\varepsilon}\| \rho_{\varepsilon}-\rho_{0} \|_{L^{\infty}_{t}H^{s'}} $ is bounded uniformly in $\varepsilon$ for $s'$ large enough. So the spatial fluctuation and its high order spatial derivative are of order $\varepsilon$. We thus believe that the restriction on the mesh size is technical. We shall investigate this question in the numerical section. Theorem \ref{main_result_2} is not fully satisfactory, it is nevertheless up to our knowledge, the first non linear discrete modulated energy estimate for the pressureless Euler-Poisson-Boltzmann equations. The rest of this section is devoted to the proofs of Theorem \ref{main_result_1} and Theorem \ref{main_result_2}. In all this section $N_{T}$ and $N$ are fixed positive integers.
\subsection{A priori estimates and existence }\label{sec:discrete_energy_estimate}

We begin with a discrete analogue of the renormalized continuity equation.

\begin{lemma}{(Discrete renormalized continuity equation)} \label{renorm_continuity_lemma} Let $ (\rho^{n},u^{n},\phi^{n})_{n = 0,\dots,N_{T}} \subset X(\mathcal{T}) \times X(\mathcal{T}^{\star}) \times X(\mathcal{T})$ be a solution to \eqref{EPBD} with  $\rho^{n} > 0$  for all $n \in \lbrace 0, ..., N_{T} \rbrace.$ Let $\psi \in \mathrm{C}^{2}(\R^{+}_{\star})$ a convex function. Then we have for $n \in \lbrace 0,\dots,N_{T}-1 \rbrace$ and $i \in \lbrace 0,\dots,N \rbrace:$
\begin{equation}
\label{renorm_continuity}
\frac{\psi(\rho^{n+1}_{i}) -\psi(\rho^{n}_i)}{\Delta t} 
+ \frac{{\cal G}_{i+\frac{1}{2}}^{n+1}-{\cal G}_{i-\frac{1}{2}}^{n+1} }{\Delta x}
-\rho_i^{n+1} u_{i+\frac{1}{2}}^{n+1} \frac{\psi'(\rho_{i+1}^{n+1})-\psi'(\rho_{i}^{n+1})}{\Delta x} =  R_{i}^{n+1} + {\cal D}_{i+\frac{1}{2}}^{n+1}
\end{equation}
with 
\begin{align}
\label{def:Gcal}
    &\mathcal{G}^{n+1}_{i+\frac{1}{2}} =  G^{n+1}(\rho_{i}^{n+1},\rho_{i+1}^{n+1},u_{i+\frac{1}{2}}^{n+1}) \psi'(\rho_{i+1}^{n+1})\\
    &{\cal D}_{i+\frac{1}{2}}^{n+1}= \frac{1}{\Delta x}(G^{n+1}(\rho_i^{n+1}, \rho_{i+1}^{n+1},u_{i+\frac{1}{2}}^{n+1})-G^{n+1}(\rho_i^{n+1}, \rho_{i}^{n+1},u_{i+\frac{1}{2}}^{n+1}))(\psi'(\rho_{i+1}^{n+1})-\psi'(\rho_{i}^{n+1}))\leq 0, \label{def:D}\\ 
   & R_{i}^{n+1} = - \psi''(\xi_{i}^{n+1})\frac{(\rho^{n}_i-\rho^{n+1}_i)^2}{2 \Delta t} \leq 0 \label{def:R},
\end{align}
where $\xi_{i}^{n+1} \in \Big( \min(\rho_{i}^{n+1},\rho_{i}^{n}), \max(\rho_{i}^{n+1},\rho_{i}^{n}) \Big).$
\end{lemma}
\begin{proof}
Let $n \in \lbrace 0,\dots,N_{T}-1 \rbrace $ and $i \in \lbrace 0,\dots,N \rbrace$. We multiply the continuity equation by $\psi'(\rho^{n+1}_i)$
to get
$$ 
\psi'(\rho^{n+1}_i)\frac{\rho^{n+1}_i -\rho^n_i}{\Delta t} +\frac{1}{\Delta x}({\cal F}^{n+1}_{i+1/2}-{\cal F}^{n+1}_{i-1/2})\psi'(\rho^{n+1}_i)=0. 
$$
Let set  
$$
A_{i}^{n+1} =\psi'(\rho^{n+1}_i)\frac{\rho^{n+1}_i -\rho^n_i}{\Delta t} \;\;  \mbox{ and } \;\; B_{i}^{n+1} = \frac{1}{\Delta x}({\cal F}^{n+1}_{i+1/2}-{\cal F}^{n+1}_{i-1/2})\psi'(\rho^{n+1}_i).
$$
Let us first focus on the temporal part $A_{i}^{n+1}$. Since $\psi  \in C^{2}(\R^{+}_{\star})$ and $\rho^{n+1} > 0$ everywhere in $\T$, a Taylor-Lagrange expansion of $\psi(\rho_{i}^{n})$ around $\rho_{i}^{n+1}$ yields the existence of $\xi_{i}^{n+1} \in \Big( \min(\rho_{i}^{n+1},\rho_{i}^{n}), \max(\rho_{i}^{n+1},\rho_{i}^{n}) \Big)$ such that
$$
\psi(\rho^{n}_i) = \psi(\rho^{n+1}_i) + \psi'(\rho^{n+1}_i)(\rho^{n}_i-\rho^{n+1}_i) + \psi''(\xi_{i}^{n+1})\frac{(\rho^{n}_i-\rho^{n+1}_i)^2}{2}, 
$$
so that 
$$
A_{i}^{n+1} =\psi'(\rho^{n+1}_i)\frac{\rho^{n+1}_i-\rho^{n}_i}{\Delta t} = \frac{\psi(\rho^{n+1}_i) -\psi(\rho^{n}_i)}{\Delta t} -R^{n+1}_i, 
$$
with $R^{n+1}_i$ given by \eqref{def:R}. 

Let us now focus on the flux part $B_{i}^{n+1}$. We omit the time dependence at this step since all the quantities are defined at the discrete time $t_{n+1}.$  We also set for ease $G_{i+\frac{1}{2}}(s,t) = G(s,t,u_{i+\frac{1}{2}}^{n+1}).$ We have,
\begin{eqnarray*}
\Delta x B_{i}^{n+1} &=& G_{i+1/2}(\rho_i, \rho_{i+1})\psi'(\rho_i) - G_{i-1/2}(\rho_{i-1}, \rho_{i})\psi'(\rho_i)\nonumber\\
&=& G_{i+1/2}(\rho_i, \rho_{i+1})\psi'(\rho_{i+1}) - G_{i-1/2}(\rho_{i-1}, \rho_{i})\psi'(\rho_i) + G_{i+1/2}(\rho_i, \rho_{i+1})(\psi'(\rho_i)-\psi'(\rho_{i+1}))\nonumber\\
&=& {\cal G}_{i+1/2}-{\cal G}_{i-1/2} - (G_{i+1/2}(\rho_i, \rho_{i+1})-G_{i+1/2}(\rho_i, \rho_{i}))(\psi'(\rho_{i+1})-\psi'(\rho_{i})) \nonumber\\
&&- G_{i+1/2}(\rho_i, \rho_{i}) (\psi'(\rho_{i+1})-\psi'(\rho_{i}))\nonumber\\
&=& {\cal G}_{i+1/2}-{\cal G}_{i-1/2} - \Delta x {\cal D}_{i+\frac{1}{2}}-\rho_i u_{i+1/2}(\psi'(\rho_{i+1})-\psi'(\rho_{i}))\nonumber
\end{eqnarray*}
where ${\cal G}_{i+1/2}$ and ${\cal D}_{i+\frac{1}{2}}$ are defined in \eqref{def:Gcal} and \eqref{def:D}. 
Note that $\mathcal{D}_{i+\frac{1}{2}}$ is non positive since $G$ is non increasing in its second variable and $\psi$ is  convex. Hence, summing $A_{i}^{n+1}$ and $B_{i}^{n+1}$ together, we finally obtain 
$$
\frac{\psi(\rho^{n+1}_{i}) -\psi(\rho^{n}_i)}{\Delta t} 
+ \frac{{\cal G}_{i+\frac{1}{2}}^{n+1}-{\cal G}_{i-\frac{1}{2}}^{n+1} }{\Delta x}
-\rho_i^{n+1} u_{i+\frac{1}{2}}^{n+1} \frac{\psi'(\rho_{i+1}^{n+1})-\psi'(\rho_{i}^{n+1})}{\Delta x} =  R_{i}^{n+1} + {\cal D}_{i+\frac{1}{2}}^{n+1}. 
$$
\end{proof}

We now give a discrete analogue of a renormalized momentum equation.
\begin{lemma} \label{renorm_momentum_lemma} Let $ (\rho^{n},u^{n},\phi^{n})_{n = 0,\dots,N_{T}} \subset X(\mathcal{T}) \times X(\mathcal{T}^{\star}) \times X(\mathcal{T})$ be a solution to \eqref{EPBD} with  $\rho^{n} > 0$ for all $n \in \lbrace 0, ..., N_{T} \rbrace.$ Let $\psi \in \mathrm{C}^{2}(\R)$ a convex function. Then we have for $n \in \lbrace 0,\dots, N_{T}-1 \rbrace$ and $i \in \lbrace 0,\dots,N \rbrace:$

\begin{equation}
\frac{ \rho_{i+\frac{1}{2}}^{n+1} \psi(u_{i+\frac{1}{2}}^{n+1}) - \rho_{i+\frac{1}{2}}^{n} \psi(u_{i+\frac{1}{2}}^{n})}{\Delta t} + \frac{\mathcal{F}_{i+1}^{n+1}\psi(u_{i+1}^{n+1}) - \mathcal{F}_{i}^{n+1}\psi(u_{i}^{n+1})}{\Delta x} 
 =  \tilde{\rho}_{i+\frac{1}{2}}^{n+1} \delta (\phi^{n+1})_{i+\frac{1}{2}} \psi'(u_{i+\frac{1}{2}}^{n+1}) - \mathcal{S}_{i}^{n+1},    
\end{equation}
with
\begin{equation} \label{def:S}
\mathcal{S}_{i}^{n+1} = - \mathcal{F}_{i+1}^{n+1} \frac{\Big( u_{i+1}^{n+1} - u_{i+\frac{1}{2}}^{n+1} \Big)^2}{2 \Delta x} \psi''( \beta_{i+\frac{1}{2}}^{n+1})
   + \mathcal{F}_{i}^{n+1} \frac{\Big( u_{i}^{n+1} - u_{i+\frac{1}{2}}^{n+1} \Big)^2}{2 \Delta x} \psi''( \gamma_{i+\frac{1}{2}}^{n+1})+\rho_{i+\frac{1}{2}}^{n}\frac{ (u_{i+\frac{1}{2}}^{n} - u_{i+\frac{1}{2}}^{n+1})^2}{2 \Delta t} \psi''( \alpha_{i+\frac{1}{2}}^{n+1}),  
\end{equation}
where $\mathcal{S}_{i}^{n+1}$ is non negative because of the definition  of the velocities at the interfaces \eqref{u_primal} and the fact that $\psi$ is a convex function.

\begin{proof} Let $i \in \lbrace 0,\dots,N \rbrace.$ We multiply the discrete momentum equation by $\psi'(u_{i+\frac{1}{2}}^{n+1})$ to get
$$
A_{i+\frac{1}{2}}^{n+1} + B_{i+\frac{1}{2}}^{n+1}  = P_{i+\frac{1}{2}}^{n+1}.
$$
with $P_{i+\frac{1}{2}}^{n+1} = \tilde{\rho}_{i+\frac{1}{2}}^{n+1} \delta (\phi^{n+1})_{i+\frac{1}{2}} \psi'(u_{i+\frac{1}{2}}^{n+1})$ and 
\begin{align*}
   & A_{i+\frac{1}{2}}^{n+1} = \frac{ \rho_{i+\frac{1}{2}}^{n+1} u_{i+\frac{1}{2}}^{n+1} - \rho_{i+\frac{1}{2}}^{n} u_{i+\frac{1}{2}}^{n} }{\Delta t} \psi'( u_{i+\frac{1}{2}}^{n+1}), \;\;\;\; B_{i+\frac{1}{2}}^{n+1} = \frac{ \mathcal{F}_{i+1}^{n+1}u_{i+1}^{n+1} - \mathcal{F}_{i}^{n+1} u_{i}^{n+1}}{\Delta x} \psi'(u_{i+\frac{1}{2}}^{n+1}). 
\end{align*}
We shall now reformulate $A_{i+\frac{1}{2}}^{n+1}$ and $B_{i+\frac{1}{2}}^{n+1}$. We have
\begin{align*}
&\Delta t A_{i+\frac{1}{2}}^{n+1} = \rho_{i+\frac{1}{2}}^{n+1} u_{i+\frac{1}{2}}^{n+1} \psi'(u_{i+\frac{1}{2}}^{n+1}) - \rho_{i+\frac{1}{2}}^{n} \big( u_{i+\frac{1}{2}}^{n} - u_{i+\frac{1}{2}}^{n+1} \big) \psi'(u_{i+\frac{1}{2}}^{n+1}) - \rho_{i+\frac{1}{2}}^{n} u_{i+\frac{1}{2}}^{n+1} \psi'(u_{i+\frac{1}{2}}^{n+1})\\
&= u_{i+\frac{1}{2}}^{n+1} \psi'(u_{i+\frac{1}{2}}^{n+1}) \Big( \rho_{i+\frac{1}{2}}^{n+1} - \rho_{i+\frac{1}{2}}^{n} \Big) - \rho_{i+\frac{1}{2}}^{n} (u_{i+\frac{1}{2}}^{n} - u_{i+\frac{1}{2}}^{n+1}) \psi'(u_{i+\frac{1}{2}}^{n+1}).
\end{align*}
A Taylor-Lagrange expansion of $\psi(u_{i+\frac{1}{2}}^{n})$ around $u_{i+\frac{1}{2}}^{n+1}$ yields the existence of $\alpha_{i+\frac{1}{2}}^{n+1} \in \Big( \min(u_{i+\frac{1}{2}}^{n},u_{i+\frac{1}{2}}^{n+1}),\max(u_{i+\frac{1}{2}}^{n},u_{i+\frac{1}{2}}^{n+1}) \Big) $ such that
$$
\psi(u_{i+\frac{1}{2}}^{n}) = \psi(u_{i+\frac{1}{2}}^{n+1}) +  \psi'(u_{i+\frac{1}{2}}^{n+1})(u_{i+\frac{1}{2}}^{n} - u_{i+\frac{1}{2}}^{n+1})  + \psi''( \alpha_{i+\frac{1}{2}}^{n+1})\frac{(u_{i+\frac{1}{2}}^{n} - u_{i+\frac{1}{2}}^{n+1})^2}{2} .
$$
Inserting this expansion in  $A_{i+\frac{1}{2}}^{n+1}$ leads to 
\begin{align*}
&\Delta t A_{i+\frac{1}{2}}^{n+1} = u_{i+\frac{1}{2}}^{n+1} \psi'(u_{i+\frac{1}{2}}^{n+1}) \Big( \rho_{i+\frac{1}{2}}^{n+1} - \rho_{i+\frac{1}{2}}^{n} \Big) -  \rho_{i+\frac{1}{2}}^{n} \Big[ \psi(u_{i+\frac{1}{2}}^{n}) - \psi(u_{i+\frac{1}{2}}^{n+1}) - \frac{ (u_{i+\frac{1}{2}}^{n} - u_{i+\frac{1}{2}}^{n+1})^{2}}{2} \psi''( \alpha_{i+\frac{1}{2}}^{n+1}) \Big]\\
& = \rho_{i+\frac{1}{2}}^{n} \psi(u_{i+\frac{1}{2}}^{n+1}) - \rho_{i+\frac{1}{2}}^{n} \psi(u_{i+\frac{1}{2}}^{n}) + u_{i+\frac{1}{2}}^{n+1} \psi'(u_{i+\frac{1}{2}}^{n+1}) (\rho_{i+\frac{1}{2}}^{n+1} - \rho_{i+\frac{1}{2}}^{n}) + \rho_{i+\frac{1}{2}}^{n} \frac{ (u_{i+\frac{1}{2}}^{n} - u_{i+\frac{1}{2}}^{n+1})^2}{2} \psi''( \alpha_{i+\frac{1}{2}}^{n+1}).
\end{align*}
The continuity equation on the dual mesh writes:
$$
\rho_{i+\frac{1}{2}}^{n+1} + \frac{\Delta t}{\Delta x} \Big( \mathcal{F}_{i+1}^{n+1} - \mathcal{F}_{i}^{n+1} \Big) = \rho_{i+\frac{1}{2}}^{n}, 
$$
hence,
\begin{align}
\label{Ademi}
    &A_{i+\frac{1}{2}}^{n+1} = \frac{ \rho_{i+\frac{1}{2}}^{n+1} \psi(u_{i+\frac{1}{2}}^{n+1}) - \rho_{i+\frac{1}{2}}^{n} \psi(u_{i+\frac{1}{2}}^{n})}{\Delta t} + \frac{ \mathcal{F}_{i+1}^{n+1} - \mathcal{F}_{i}^{n+1}}{\Delta x} \psi(u_{i+\frac{1}{2}}^{n+1}) + u_{i+\frac{1}{2}}^{n+1} \psi'(u_{i+\frac{1}{2}}^{n+1}) \frac{ \rho_{i+\frac{1}{2}}^{n+1} - \rho_{i+\frac{1}{2}}^{n}}{\Delta t} \\
    &+\rho_{i+\frac{1}{2}}^{n}\frac{ (u_{i+\frac{1}{2}}^{n} - u_{i+\frac{1}{2}}^{n+1})^2}{2 \Delta t} \psi''( \alpha_{i+\frac{1}{2}}^{n+1}). \nonumber
\end{align}
We now deal with $B_{i+\frac{1}{2}}^{n+1}$ which we rewrite as 
\begin{align*}
    & \Delta x B_{i+\frac{1}{2}}^{n+1} = \mathcal{F}_{i+1}^{n+1} \Big( u_{i+1}^{n+1} - u_{i+\frac{1}{2}}^{n+1} \Big) \psi'\big( u_{i+\frac{1}{2}}^{n+1} \big) - \mathcal{F}_{i}^{n+1} \Big( u_{i}^{n+1} - u_{i+\frac{1}{2}}^{n+1} \Big) \psi'\big( u_{i+\frac{1}{2}}^{n+1} \big) \\
    & + u_{i+\frac{1}{2}}^{n+1} \psi'( u_{i+\frac{1}{2}}^{n+1}) \Big( \mathcal{F}_{i+1}^{n+1} - \mathcal{F}_{i}^{n+1} \Big).
\end{align*}
The Taylor-Lagrange expansions of $\psi$ yield the existence of $\beta_{i+\frac{1}{2}}^{n+1} \in \Big( \min(u_{i+1}^{n+1},u_{i+\frac{1}{2}}^{n+1}), \max(u_{i+1}^{n+1},u_{i+\frac{1}{2}}^{n+1}) \Big)$ and $\gamma_{i+\frac{1}{2}}^{n+1} \in \Big( \min(u_{i}^{n+1},u_{i+\frac{1}{2}}^{n+1}), \max(u_{i}^{n+1},u_{i+\frac{1}{2}}^{n+1}) \Big)$ such that
\begin{eqnarray*}
    \psi(u_{i+1}^{n+1}) &=& \psi(u_{i+\frac{1}{2}}^{n+1}) + \Big( u_{i+1}^{n+1} - u_{i+\frac{1}{2}}^{n+1} \Big) \psi'(u_{i+\frac{1}{2}}^{n+1}) +\frac{\Big( u_{i+1}^{n+1} - u_{i+\frac{1}{2}}^{n+1} \Big)^2}{2} \psi''( \beta_{i+\frac{1}{2}}^{n+1}),\nonumber\\
    \psi(u_{i}^{n+1}) &=& \psi(u_{i+\frac{1}{2}}^{n+1}) + \Big( u_{i}^{n+1} - u_{i+\frac{1}{2}}^{n+1} \Big) \psi'(u_{i+\frac{1}{2}}^{n+1}) +\frac{\Big( u_{i}^{n+1} - u_{i+\frac{1}{2}}^{n+1} \Big)^2}{2} \psi''( \gamma_{i+\frac{1}{2}}^{n+1}).\nonumber
\end{eqnarray*}
Inserting these expressions in $B_{i+\frac{1}{2}}^{n+1}$ leads to 
\begin{align*}
    &\Delta x B_{i+\frac{1}{2}}^{n+1} = \mathcal{F}_{i+1}^{n+1} \Big[ \psi(u_{i+1}^{n+1}) -\psi(u_{i+\frac{1}{2}}^{n+1}) - \frac{\Big( u_{i+1}^{n+1} - u_{i+\frac{1}{2}}^{n+1} \Big)^2}{2} \psi''( \beta_{i+\frac{1}{2}}^{n+1}) \Big]\\
    & - \mathcal{F}_{i}^{n+1} \Big[  \psi(u_{i}^{n+1}) - \psi(u_{i+\frac{1}{2}}^{n+1}) - \frac{\Big( u_{i}^{n+1} - u_{i+\frac{1}{2}}^{n+1} \Big)^2}{2} \psi''( \gamma_{i+\frac{1}{2}}^{n+1})  \Big]
    + u_{i+\frac{1}{2}}^{n+1} \psi'( u_{i+\frac{1}{2}}^{n+1}) \Big( \mathcal{F}_{i+1}^{n+1} - \mathcal{F}_{i}^{n+1} \Big).
\end{align*}
Rearranging the terms we eventually obtain
\begin{align}
    \Delta x B_{i+\frac{1}{2}}^{n+1} &= \mathcal{F}_{i+1}^{n+1}\psi(u_{i+1}^{n+1}) - \mathcal{F}_{i}^{n+1}\psi(u_{i}^{n+1}) - \psi(u_{i+\frac{1}{2}}^{n+1}) \Big( \mathcal{F}_{i+1}^{n+1} - \mathcal{F}_{i}^{n+1} \Big)+ u_{i+\frac{1}{2}}^{n+1} \psi'(u_{i+\frac{1}{2}}^{n+1})\Big( \mathcal{F}_{i+1}^{n+1} - \mathcal{F}_{i}^{n+1} \Big) \nonumber \\
   \label{Bdemi}
   &- \mathcal{F}_{i+1}^{n+1} \frac{\Big( u_{i+1}^{n+1} - u_{i+\frac{1}{2}}^{n+1} \Big)^2}{2} \psi''( \beta_{i+\frac{1}{2}}^{n+1})
   + \mathcal{F}_{i}^{n+1} \frac{\Big( u_{i}^{n+1} - u_{i+\frac{1}{2}}^{n+1} \Big)^2}{2} \psi''( \gamma_{i+\frac{1}{2}}^{n+1}).
\end{align}
Using eventually the continuity equation on the dual mesh  
$$
\frac{\rho_{i+\frac{1}{2}}^{n+1} - \rho_{i+\frac{1}{2}}^{n}}{\Delta t } + \frac{\mathcal{F}_{i+1}^{n+1} - \mathcal{F}_{i}^{n+1}}{\Delta x} = 0,
$$
we get 
\begin{eqnarray*}
    \Delta x B_{i+\frac{1}{2}}^{n+1} &=& \mathcal{F}_{i+1}^{n+1}\psi(u_{i+1}^{n+1}) - \mathcal{F}_{i}^{n+1}\psi(u_{i}^{n+1}) - \psi(u_{i+\frac{1}{2}}^{n+1}) \Big( \mathcal{F}_{i+1}^{n+1} - \mathcal{F}_{i}^{n+1} \Big)- u_{i+\frac{1}{2}}^{n+1} \psi'(u_{i+\frac{1}{2}}^{n+1})(\rho_{i+\frac{1}{2}}^{n+1} - \rho_{i+\frac{1}{2}}^{n})\frac{\Delta x}{\Delta t} \nonumber \\
   &&- \mathcal{F}_{i+1}^{n+1} \frac{\Big( u_{i+1}^{n+1} - u_{i+\frac{1}{2}}^{n+1} \Big)^2}{2} \psi''( \beta_{i+\frac{1}{2}}^{n+1})
   + \mathcal{F}_{i}^{n+1} \frac{\Big( u_{i}^{n+1} - u_{i+\frac{1}{2}}^{n+1} \Big)^2}{2} \psi''( \gamma_{i+\frac{1}{2}}^{n+1}).\nonumber 
\end{eqnarray*}
Finally, we glean $A_{i+\frac{1}{2}}^{n+1}$ and $B_{i+\frac{1}{2}}^{n+1}$ 
\begin{align} \label{AplusB}
    &A_{i+\frac{1}{2}}^{n+1} + B_{i+\frac{1}{2}}^{n+1} = \frac{ \rho_{i+\frac{1}{2}}^{n+1} \psi(u_{i+\frac{1}{2}}^{n+1}) - \rho_{i+\frac{1}{2}}^{n} \psi(u_{i+\frac{1}{2}}^{n})}{\Delta t} + \frac{\mathcal{F}_{i+1}^{n+1}\psi(u_{i+1}^{n+1}) - \mathcal{F}_{i}^{n+1}\psi(u_{i}^{n+1})}{\Delta x} \\
    &- \mathcal{F}_{i+1}^{n+1} \frac{\Big( u_{i+1}^{n+1} - u_{i+\frac{1}{2}}^{n+1} \Big)^2}{2 \Delta x} \psi''( \beta_{i+\frac{1}{2}}^{n+1})
   + \mathcal{F}_{i}^{n+1} \frac{\Big( u_{i}^{n+1} - u_{i+\frac{1}{2}}^{n+1} \Big)^2}{2 \Delta x} \psi''( \gamma_{i+\frac{1}{2}}^{n+1}) +\rho_{i+\frac{1}{2}}^{n}\frac{ (u_{i+\frac{1}{2}}^{n} - u_{i+\frac{1}{2}}^{n+1})^2}{2 \Delta t} \psi''( \alpha_{i+\frac{1}{2}}^{n+1}). \nonumber
\end{align}
We obtained the expected result.
\end{proof}
\end{lemma}
A consequence of Lemmas \ref{renorm_continuity_lemma}, \ref{renorm_momentum_lemma} and the definition \eqref{rho_upwind} is that the discrete energy functional \eqref{def:H_discret} decays along the solutions of the scheme \eqref{EPBD}. More precisely we have.
\begin{pro}{(Discrete energy decay)} \label{energy_decay} Let $ (\rho^{n},u^{n},\phi^{n})_{n = 0,\dots,N_{T}} \subset X(\mathcal{T}) \times X(\mathcal{T}^{\star}) \times X(\mathcal{T})$ a solution to \eqref{EPBD} with  $\rho^{n} > 0$ for all $n \in \lbrace 0, ..., N_{T} \rbrace$. Assume $\phi^{0} \in X(\mathcal{T})$ satisfies the discrete Poisson equation initially.
Then we have 

\begin{equation}
 \forall n \in \lbrace 0,\dots,N_{T}-1 \rbrace \quad  \mathcal{H}\Big( \rho^{n+1},u^{n+1},\phi^{n+1} \Big) - \mathcal{H}\Big( \rho^{n},u^{n},\phi^{n}\Big) =  - \Delta t\tau^{n+1},
\end{equation}
where $\tau^{n+1} \geq 0$ is given by
\begin{align} \label{tau_n}
   & \tau^{n+1} =  \sum_{i=0}^{N} \Big( \mathcal{S}_{i}^{n+1} + e^{-\xi^n_i}\frac{(\phi_i^{n+1}-\phi_i^{n})^2}{2 \Delta t} \Big) \Delta x +\frac{\varepsilon^2}{2\Delta t} \sum_{i=0}^{N} 
\Big(\frac{\phi_{i+1}^{n+1}-\phi^{n+1}_i}{\Delta x} - \frac{\phi_{i+1}^{n}-\phi^{n}_i}{\Delta x}\Big)^2 \Delta x\\
&  +\frac{g(0)}{\Delta x} \sum_{i=0}^{N} \varepsilon^2 \big \vert \delta (\phi^{n+1})_{i+\frac{1}{2}}-\delta (\phi^{n+1})_{i-\frac{1}{2}} \big \vert^2 \Delta x + \frac{g(0)}{\Delta x} \sum_{i=0}^{N} \big \vert (e^{-\phi_{i}^{n+1}} - e^{-\phi_{i-1}^{n+1}})(\phi_{i}^{n+1}-\phi_{i-1}^{n+1}) \big \vert \Delta x \nonumber
\end{align}
where $S_{i}^{n+1} \geq 0$ is given \eqref{def:S} with $\psi : s\in \R \longmapsto \frac{s^2}{2}$ and for all i $\in \lbrace 0,\dots,N \rbrace$, $\xi^n_i\in \Big( \min(\phi_i^n, \phi_i^{n+1}), \max(\phi_i^n, \phi_i^{n+1})  \Big)$.

\begin{proof} Let $n \in \lbrace 0,\dots, N_{T}-1 \rbrace.$ Observe on the one hand that applying Lemma \ref{renorm_momentum_lemma} with the function $\psi : s \in \R \longmapsto \frac{s^{2}}{2}$ yields for all $i \in \lbrace 0,\dots,N \rbrace,$
\begin{align}\label{discrete_kinetic_energy_balance}
&\frac{ \rho_{i+\frac{1}{2}}^{n+1} (u_{i+\frac{1}{2}}^{n+1})^2 - \rho_{i+\frac{1}{2}}^{n} (u_{i+\frac{1}{2}}^{n})^2}{2\Delta t} + \frac{\mathcal{F}_{i+1}^{n+1}(u_{i+1}^{n+1})^2 - \mathcal{F}_{i}^{n+1}(u_{i}^{n+1})^2}{2\Delta x}  =  \tilde{\rho}_{i+\frac{1}{2}}^{n+1} u_{i+\frac{1}{2}}^{n+1} \delta (\phi^{n+1})_{i+\frac{1}{2}}  - \mathcal{S}_{i}^{n+1}  
\end{align}
where $\mathcal{S}_{i}^{n+1} \geq 0$ is given by \eqref{def:S}. Let us treat the first term of the right hand side. Using the definition \eqref{rho_upwind} we have
$$
D:= \sum_{i=0}^{N} \tilde{\rho}_{i+\frac{1}{2}}^{n+1} u_{i+\frac{1}{2}}^{n+1} \delta (\phi^{n+1})_{i+\frac{1}{2}} \Delta x = \sum_{i=0}^{N} \mathcal{F}_{i+\frac{1}{2}}^{n+1} \delta (\phi^{n+1})_{i+\frac{1}{2}} \Delta x 
 - \sum_{i=0}^{N} G\Big(\rho_{i}^{n+1},\rho_{i+1}^{n+1},0\Big)\delta (\phi^{n+1})_{i+\frac{1}{2}} \Delta x.
$$
Let us set 
$$
D_{1} = \sum_{i=0}^{N} \mathcal{F}_{i+\frac{1}{2}}^{n+1} \delta (\phi^{n+1})_{i+\frac{1}{2}} \Delta x,$$  and  
$$
D_{2} = - \sum_{i=0}^{N} G(\rho_{i}^{n+1},\rho_{i+1}^{n+1},0)\delta( \phi^{n+1})_{i+\frac{1}{2}} \Delta x = g(0)\sum_{i=0}^{N} (\rho_{i+1}^{n+1}-\rho_{i}^{n+1})\delta (\phi^{n+1})_{i+\frac{1}{2}} \Delta x. 
$$
As for the first term, using the discrete integration by parts \eqref{ipp1} of Lemma \ref{ipp_lemma} and the continuity equation, we have
$$
D_{1} = - \sum_{i=0}^{N} \frac{ \mathcal{F}_{i+\frac{1}{2}}^{n+1} - \mathcal{F}_{i-\frac{1}{2}}^{n+1} }{\Delta x} \phi_{i}^{n+1}  \Delta x= \sum_{i=0}^{N} \frac{\rho_{i}^{n+1} - \rho_{i}^{n}}{\Delta t} \phi_{i}^{n+1} \Delta x. 
$$
Then, using the discrete Poisson equation (which also holds initially), we obtain
\begin{align*}
&\sum_{i=0}^{N} \frac{\rho_{i}^{n+1} - \rho_{i}^{n}}{\Delta t} \phi_{i}^{n+1} \Delta x= \frac{\varepsilon^2 }{\Delta t} \sum_{i=0}^{N} \Big( \Delta (\phi^{n+1})_{i} - \Delta(\phi^{n})_{i}\Big)\phi^{n+1}_i \Delta x + \frac{1}{\Delta t} \sum_{i=0}^{N} \Big(e^{-\phi_i^{n+1}} - e^{-\phi_i^{n}}  \Big)\phi^{n+1}_i \Delta x\nonumber\\
&= -\frac{\varepsilon^2}{\Delta t} \sum_{i=0}^{N} \Big( \delta (\phi^{n+1})_{i+\frac{1}{2}}- \delta( \phi^{n})_{i+\frac{1}{2}}\Big) \delta (\phi^{n+1})_{i+\frac{1}{2}} \Delta x  + \frac{1}{\Delta t} \sum_{i=0}^{N} \Big(e^{-\phi_i^{n+1}} - e^{-\phi_i^{n}}\Big)\phi^{n+1}_i \Delta x, \nonumber
\end{align*} 
where we used the discrete integration by parts \eqref{ipp2}.
We now use the identity $-a(a-b)=-a^2/2+b^2/2-(a-b)^2/2$ for $a, b\in \mathbb{R}$, so the first term re-writes 
\begin{align*}
&-\frac{\varepsilon^2}{\Delta t} \sum_{i=0}^{N} \Big( \delta (\phi^{n+1})_{i+\frac{1}{2}}- \delta (\phi^{n})_{i+\frac{1}{2}}\Big) (\delta \phi^{n+1})_{i+\frac{1}{2}} \Delta x\\
&= -\frac{\varepsilon^2}{2\Delta t} \sum_{i=0}^{N} \vert \delta (\phi^{n+1})_{i+\frac{1}{2}} \vert^2\Delta x
+\frac{\varepsilon^2}{2\Delta t} \sum_{i=0}^{N} \vert\delta( \phi^{n})_{i+\frac{1}{2}} \vert^2 \Delta x- \frac{\varepsilon^2}{2\Delta t} \sum_{i=0}^{N} 
\Big(\delta (\phi^{n+1})_{i+\frac{1}{2}} - \delta (\phi^{n})_{i+\frac{1}{2}} \Big)^2 \Delta x. 
\end{align*} 
As for the second term we have
\begin{align*}
&\frac{1}{\Delta t} \sum_{i=0}^{N}(e^{-\phi_i^{n+1}}-e^{-\phi_i^{n}}  )(\phi_i^{n+1}+1-\phi^{n}_i-1+\phi^{n}_i) \Delta x \nonumber\\
&\hspace{0.5cm}=  -\frac{1}{\Delta t} \sum_{i=0}^{N} h(\phi_i^{n+1})\Delta x+\frac{1}{\Delta t} \sum_{i=0}^{N} h(\phi_i^{n}) \Delta x  -\frac{1}{\Delta t} \sum_{i=0}^{N}\Big(e^{-\phi_i^{n+1}} - e^{-\phi_i^{n}} +  e^{-\phi_i^{n}}(\phi_i^{n+1}-\phi_i^{n})\Big) \Delta x \nonumber\\
&\hspace{0.5cm}=-\frac{1}{\Delta t} \sum_{i=0}^{N} h(\phi_i^{n+1}) \Delta x+\frac{1}{\Delta t} \sum_{i=0}^{N} h(\phi_i^{n}) \Delta x - \frac{1}{\Delta t} \sum_{i=0}^{N} e^{-\xi^n_i}\frac{(\phi_i^{n+1}-\phi_i^{n})^2}{2} \Delta x,
\end{align*}
where Taylor-Lagrange expansion has been performed with $\xi^n_i\in \Big( \min(\phi_i^n, \phi_i^{n+1}), \max(\phi_i^n, \phi_i^{n+1})  \Big)$.
Gathering the terms together we obtain
\begin{align*}
 \Delta t D_1 &= - \sum_{i=0}^{N} \frac{\varepsilon^2}{2} \big \vert \delta (\phi^{n+1})_{i+\frac{1}{2}} \big \vert^2 \Delta x+ \sum_{i=0}^{N} \frac{\varepsilon^2}{2} \big \vert \delta (\phi^{n})_{i+\frac{1}{2}} \big \vert^2 \Delta x \nonumber\\
 &-\sum_{i=0}^{N} h(\phi_i^{n+1}) \Delta x+\sum_{i=0}^{N} h(\phi_i^{n}) \Delta x  \nonumber\\
 &- \frac{\varepsilon^2}{2} \sum_{i=0}^{N} 
\Big(\delta (\phi^{n+1})_{i+\frac{1}{2}} - \delta( \phi^{n})_{i+\frac{1}{2}} \Big)^2 \Delta x
 - \sum_{i=0}^{N} e^{-\xi^n_i}\frac{(\phi_i^{n+1}-\phi_i^{n})^2}{2} \Delta x. 
\end{align*}
We now treat the second term $D_2$:
\begin{align*}
D_{2} &=  g(0) \sum_{i= 0}^{N} (\rho_{i+1}^{n+1} - \rho_{i}^{n+1}) \delta (\phi^{n+1})_{i+\frac{1}{2}}\Delta x  = -\frac{g(0)}{\Delta x} \sum_{i=0}^{N} \rho_{i}^{n+1} (\phi_{i+1}^{n+1} -2 \phi_{i}^{n+1} + \phi_{i-1}^{n+1})\Delta x, 
\end{align*}
where we used a discrete integration by parts to get the last equality. Then using the discrete Poisson equation we get
\begin{align*}
&D_2 = - \frac{g(0)}{\Delta x} \sum_{i=0}^{N} \varepsilon^2 \big \vert \delta (\phi^{n+1})_{i+\frac{1}{2}}-\delta (\phi^{n+1})_{i-\frac{1}{2}} \big \vert^2 \Delta x \nonumber - \frac{g(0)}{\Delta x} \sum_{i=0}^{N} e^{-\phi_{i}^{n+1}}( \phi_{i+1}^{n+1} - \phi_{i}^{n+1} -(\phi_{i}^{n+1}-\phi_{i-1}^{n+1}))\Delta x.
\end{align*}
Since $g(0) \geq 0$ the first sum is non positive. The second sum is also non positive since after using a change of indices we have
\begin{align*}
-\sum_{i=0}^{N} e^{-\phi_{i}^{n+1}}( \phi_{i+1}^{n+1} - \phi_{i}^{n+1} -(\phi_{i}^{n+1}-\phi_{i-1}^{n+1})) \Delta x = \sum_{i=0}^{N} (e^{-\phi_{i}^{n+1}} - e^{-\phi_{i-1}^{n+1}})(\phi_{i}^{n+1}-\phi_{i-1}^{n+1}) \Delta x.    
\end{align*}
Since the function $s \longmapsto e^{-s}$ is decreasing we thus get $ \sum_{i=0}^{N} (e^{-\phi_{i}^{n+1}} - e^{-\phi_{i-1}^{n+1}})(\phi_{i}^{n+1}-\phi_{i-1}^{n+1})\Delta x \leq 0.$
Thus, it yields for $D_2$
\begin{align*}
D_2 &= - \frac{g(0)}{\Delta x} \sum_{i=0}^{N} \varepsilon^2 \big \vert \delta (\phi^{n+1})_{i+\frac{1}{2}}-\delta (\phi^{n+1})_{i-\frac{1}{2}} \big \vert^2 \Delta x - \frac{g(0)}{\Delta x} \sum_{i=0}^{N}   \big \vert (e^{-\phi_{i}^{n+1}} - e^{-\phi_{i-1}^{n+1}})(\phi_{i}^{n+1}-\phi_{i-1}^{n+1} )\big \vert \Delta x. 
\end{align*}
So we eventually obtain,
\begin{align*}
\Delta t D :=\Delta t \sum_{i=0}^{N}\tilde{\rho}_{i+\frac{1}{2}}^{n+1} u_{i+\frac{1}{2}}^{n+1} \delta (\phi^{n+1})_{i+\frac{1}{2}} \Delta x = \Delta t D_{1} + \Delta t D_{2}, 
\end{align*}
which enables to recover the terms in \eqref{tau_n}. Summing \eqref{discrete_kinetic_energy_balance} over $i \in \lbrace 0,\dots,N \rbrace$ then leads to the discrete kinetic energy balance.
\end{proof}
\end{pro}

We now establish  discrete elliptic estimates for the discrete Poisson equation. 

\begin{pro}{(Discrete elliptic estimates)} \label{discrete_elliptic_estimates}
Let $ (\rho^{n},u^{n},\phi^{n})_{n = 0,\dots,N_{T}} \subset X(\mathcal{T})\times X(\mathcal{T}^{\star}) \times X(\mathcal{T})$ a solution to \eqref{EPBD} with  $\rho^{n} > 0$ for all $n \in \lbrace 0, ..., N_{T} \rbrace.$ Assume $\phi^{0} \in X(\mathcal{T})$ satisfies the discrete Poisson equation initially.   Then the discrete potential satisfies for all $n \in \lbrace 0,\dots,N_{T} \rbrace$ the following estimates:
\begin{align}
    &\forall p \in [1,+\infty), \quad \Big \| e^{-\phi^{n}_{}} \Big \|_{L^{p}(\T)} \leq \| \rho^{n} \|_{L^{p}(\T)}, \label{estimate_exp_lp}\\
    &\forall p \in [1,+\infty), \quad \varepsilon^2 \sum_{i=0}^{N} \Big | \frac{ \phi_{i+1}^{n} - \phi_{i}^{n}}{\Delta x}\frac{ e^{-(p-1)\phi_{i+1}^{n}} - e^{-(p-1)\phi_{i}^{n}}}{ \Delta x}  \Big |\Delta x \leq \| \rho^{n} \|^p_{L^{p}(\T)}, \label{estimate_dx_exp_phi}\\
    & \forall i \in \lbrace 0,\dots,N \rbrace, \: \underset{ i \in \lbrace 0,\dots,N \rbrace}{\min \log(\rho_{i}^{n})} \leq - \phi^{n}_{i} \leq \underset{ i \in \lbrace 0,\dots,N \rbrace}{\max \log( \rho_{i}^{n})} \label{maximum_principle},\\
    & \underset{ i \in \lbrace 0,\dots,N \rbrace}{\min \log(\rho_{i}^{n})} \leq \langle - \phi^{n} \rangle \leq \underset{ i \in \lbrace 0,\dots,N \rbrace}{\max \log( \rho_{i}^{n})}, \label{estim_mean_value_phi} \\
    & \alpha^{n} \Big \| \phi^{n} - \langle \phi^{n} \rangle \Big \|_{L^{2}(\T)} \leq \Vert \rho^{n} \Vert_{L^{2}(\T)} +  \Vert \rho^{n} \Vert_{L^{1}(\T)}, \label{l2_estim_phi}\\
   & \varepsilon \Big  | \phi^{n} - \langle \phi^{n} \rangle \Big |_{H^{1}(\T)} \leq \frac{1}{\sqrt{\alpha^{n}}} \Big( \| \rho^{n} \|_{L^{2}(\T)}  +  \| \rho^{n} \|_{L^{1}(\T)} \Big)\label{estimate_semi_h1},
\end{align}
with $\langle \phi^n\rangle =\sum_{i=0}^N \phi^n_i \Delta x$, $\alpha^{n} =  e^{\underset{ i \in \lbrace 0,\dots,N \rbrace}{\min \log(\rho_{i}^{n})} }$ and where the inequality \eqref{estimate_exp_lp} is an equality for $p = 1$. 
\begin{proof}
We prove \eqref{estimate_exp_lp}. Fix  $p \in [1,+\infty)$ and $ n \in \lbrace 0,\dots,N_{T} \rbrace.$ We set $\psi = -\phi^{n} $. We multiply the discrete Poisson equation  by $e^{(p-1) \psi_{i}}$ for $i \in \lbrace 0,\dots,N\rbrace$ and sum over $i \in \lbrace 0,\dots,N \rbrace.$  We  get after a discrete integration by parts
\begin{align}
\sum_{i=0}^{N} \varepsilon^2  \frac{ \psi_{i+1}- \psi_{i}}{\Delta x}  \frac{ e^{(p-1)\psi_{i+1}} - e^{(p-1)\psi_{i}} }{\Delta x} \Delta x+ \sum_{i= 0}^{N} e^{ p\psi_{i}} \Delta x = \sum_{i=0}^{N} \rho_{i}^{n} e^{(p-1) \psi_{i}} \Delta x. \label{roni}
\end{align}
Since $p \geq 1$ and the exponential function is increasing, the first sum is non negative. As a result we obtain on the one hand
\begin{align*}
\sum_{i=0}^{N} \vert e^{\psi_{i}} \vert ^{p} \Delta x\leq \sum_{i=0}^{N} \rho_{i}^{n} e^{(p-1) \psi_{i}} \Delta x \leq  \Big( \sum_{i=0}^{N} \vert \rho_{i}^{n} \vert ^{p}  \Delta x\Big)^{\frac{1}{p}} \Big( \sum_{i=0}^{N}  \vert e^{\psi_{i}} \vert^{p}  \Delta x\Big)^{\frac{p-1}{p}}, 
\end{align*}
where for the last inequality we used the Hölder inequality in the duality $(\ell^{p}(\R^{N+1}), \ell^{\frac{p}{p-1}}(\R^{N+1}))$. It yields \eqref{estimate_exp_lp} after a simplification. 
The case $p = 1$ yields the equality in \eqref{estimate_exp_lp} directly from \eqref{roni}. Using again \eqref{roni} we obtain on the other hand,
\begin{align*}
\sum_{i=0}^{N} \varepsilon^2  \frac{ \psi_{i+1}- \psi_{i}}{\Delta x}  \frac{ e^{(p-1)\psi_{i+1}} - e^{(p-1)\psi_{i}} }{\Delta x} \Delta x \leq \| \rho^{n} \|_{L^{p}(\T)} \| e^{-\phi^{n}} \|_{L^{p}(\T)}^{p-1} \leq \| \rho^{n} \|_{L^{p}(\T)}^{p}.
\end{align*}
It proves \eqref{estimate_dx_exp_phi}.  We prove the maximum principle \eqref{maximum_principle}. We set $M = \underset{ i \in \lbrace 0,\dots,N \rbrace} \max  \log(\rho_{i}^{n})$ and $\psi = -\phi^{n} .$ Then for all $i \in \lbrace 0,\dots,N\rbrace:$
\begin{align}
    -\varepsilon^2 \frac{ \psi_{i+1} - 2 \psi_{i} + \psi_{i-1} }{\Delta x^2} + e^{\psi_{i}} - e^{M} = \rho_{i}^{n} - e^{M}. \label{tito}
\end{align}
By definition of $M$ we have $\rho_{i}^{n} -e^{M} \leq 0$ for all $i \in \lbrace 0,\dots,N \rbrace.$ Let us now prove the claim by contradiction. Assume that there exists $j \in \lbrace 0,\dots,N \rbrace$ such that $\psi_{j} > M.$ Then let $i^{\star} \in \lbrace 0,\dots, N \rbrace$ such that $\psi_{i^{\star}} = \underset{i \in \lbrace 0,\dots,N \rbrace } \max \psi_{i}.$ Then we have 
$$
\psi_{i^{\star}} > M \Longrightarrow e^{\psi_{i^{\star}}} > e^{M}.
$$
In addition, 
$$
-\varepsilon^2 \frac{ \psi_{i^{\star}+1} - 2 \psi_{i^{\star}} + \psi_{i^{\star}-1} }{\Delta x^2} \geq 0
$$
so we deduce from \eqref{tito} applied at $i = i^{\star}$ that $\rho_{i^{\star}} > e^{M}$ which is the expected contradiction. A similar reasoning with $m = \underset{ i \in \lbrace 0,\dots,N \rbrace} \min  \log(\rho_{i}^{n})$ yields the lower bound in \eqref{maximum_principle}. The estimate \eqref{estim_mean_value_phi} is an immediate consequence of \eqref{maximum_principle} since $\displaystyle \sum_{i=0}^{N} \Delta x = 1.$ We now prove the $L^{2}(\T)$ estimate \eqref{l2_estim_phi}.
We set here $\psi = -\phi^{n} + \langle \phi^{n} \rangle.$ Then we have for $i \in \lbrace 0,\dots,N\rbrace:$
\begin{align}
    -\varepsilon^2 \frac{ \psi_{i+1} - 2 \psi_{i} + \psi_{i-1} }{\Delta x^2} + e^{-\phi_{i}^{n}} - e^{\langle \phi^{n} \rangle} = \rho_{i}^{n} - e^{\langle \phi^{n} \rangle}. \label{zlatan}
\end{align}
We multiply \eqref{zlatan} by $\psi_{i}$ for all $i \in \lbrace 0,\dots,N\rbrace$ and sum. Then after a discrete integration by parts we have 
\begin{align}
    \sum_{i=0}^{N} \varepsilon^2 \Big \vert \frac{\psi_{i+1}-\psi_{i}}{\Delta x} \Big \vert^2 \Delta x + \sum_{i=0}^{N} \big( e^{-\phi_{i}^{n}} - e^{\langle \phi^{n} \rangle }\big) \Big( -\phi_{i}^{n} + \langle \phi^{n} \rangle \Big) \Delta x = \sum_{i=0}^{N} \Big( \rho_{i}^{n} -e^{\langle \phi^{n} \rangle } \Big) \Big( -\phi_{i}^{n} + \langle \phi^{n} \rangle \Big) \Delta x.
\end{align}
Using the mean value theorem, we have for each $i \in \lbrace 0,\dots,N \rbrace$ the existence of $\xi_{i}^{n} \in \Big( \min(-\phi_{i}^{n},\langle \phi^{n} \rangle), \max(-\phi_{i}^{n},\langle \phi^{n} \rangle) \Big) $ such that
$$
\big( e^{-\phi_{i}^{n}} - e^{\langle \phi^{n} \rangle }\big) \Big( -\phi_{i}^{n} + \langle \phi^{n} \rangle \Big) = e^{\xi_{i}^{n}} \Big( -\phi_{i}^{n} + \langle \phi^{n} \rangle \Big)^2.
$$
Because of the bounds \eqref{maximum_principle}, \eqref{estim_mean_value_phi}, we deduce that $e^{\xi_{i}^{n}} \geq e^{\underset{ i \in \lbrace 0,\dots,N \rbrace}{\min \log(\rho_{i}^{n})} }:=\alpha^{n}$ so that this lower bound is uniform with respect to $i$. We thus obtain
\begin{align}
    \sum_{i=0}^{N} \varepsilon^2 \Big \vert \frac{\psi_{i+1}-\psi_{i}}{\Delta x} \Big \vert^2 \Delta x + \alpha^{n} \sum_{i=0}^{N} \Big \vert -\phi_{i}^{n} + \langle \phi^{n} \rangle \Big \vert^2  \Delta x \leq \sum_{i=0}^{N} \Big( \rho_{i}^{n} -e^{ \langle \phi^{n} \rangle } \Big) \Big( -\phi_{i}^{n} + \langle \phi^{n} \rangle \Big) \Delta x. \label{nedved}
\end{align}
Besides, using a Cauchy-Schwarz inequality for the right hand side we deduce the first bound
\begin{align}
\label{ineg_phi-avphi}
    &\alpha^{n} \Big \Vert \phi^{n} - \langle \phi^{n} \rangle \Big \Vert_{L^{2}(\T)} \leq \Big \Vert \rho^{n} - e^{\langle \phi^{n} \rangle} \Big \Vert_{L^{2}(\T)} \leq \Vert \rho^{n} \Vert_{L^{2}(\T)}  + e^{\langle \phi^{n} \rangle}\\
    &\leq \Vert \rho^{n} \Vert_{L^{2}(\T)} + \| e^{ \phi^n} \|_{L^{1}(\T)} \leq \Vert \rho^{n} \Vert_{L^{2}(\T)} +  \Vert \rho^{n} \Vert_{L^{1}(\T)}.
\end{align}
where we used Jensen's inequality and the bound \eqref{estimate_exp_lp}. 
We now prove the last estimate \eqref{estimate_semi_h1}. From \eqref{nedved}, we have (recalling $\psi_i=-\phi_i^n+\langle \phi^n\rangle$)
\begin{align}
    &\varepsilon^2 | \phi^{n} - \langle \phi^{n} \rangle |_{H^{1}(\T)}^2  \leq \|\rho^n - e^{\langle \phi^n\rangle}\|_{L^{2}(\T)} \|\phi^n- \langle \phi^n\rangle \|_{L^{2}(\T)}   \\
    &\leq \Big(\Vert \rho^{n} \Vert_{L^{2}(\T)} +  \Vert \rho^{n} \Vert_{L^{1}(\T)}\Big) \| \phi^{n} - \langle \phi^{n} \rangle \|_{L^{2}(\T)} \\
    &\leq \frac{1}{\alpha^{n} } \Big(\Vert \rho^{n} \Vert_{L^{2}(\T)} +  \Vert \rho^{n} \Vert_{L^{1}(\T)}\Big)^2.
\end{align}
where we used \eqref{l2_estim_phi} for the last inequality. It yields \eqref{estimate_semi_h1}.
\end{proof}
\end{pro}
For the existence proof, we will need a weaker $H^{1}$ estimate than \eqref{estimate_semi_h1} for the discrete potential which does not involve the constant $\alpha^{n}$ given in Proposition \ref{discrete_elliptic_estimates}. In this respect, we will use the discrete analogue of the Poincaré-Wirtinger inequality in $L^{2}(\T)$ that we prove right after.
\begin{lemma}{(Discrete Poincaré-Wirtinger inequality)}
\begin{align} \label{discrete_PW_ineq}
\forall u \in X(\mathcal{T}), \quad \Big \Vert u - \langle u \rangle \Big \Vert_{L^{2}(\T)}^2 \leq  \frac{1}{3} \big | u \big |_{H^{1}(\T)}^{2},
\end{align}
where $\displaystyle \langle u \rangle = \sum_{i=0}^{N} u_{i} \Delta x.$
\begin{proof} The proof mimicks the one in the continuous case. Let $u \in X(\mathcal{T}).$ We have
\begin{align}
&\sum_{i=0}^{N} \Big \vert u_{i} -\langle u \rangle \Big \vert^2 \Delta x = \sum_{i=0}^{N} \Big \vert u_{i} - \sum_{j=0}^{N} u_{j} \Delta x \Big \vert ^2 \Delta x \nonumber \\
&\leq \sum_{i=0}^{N} \Big \vert  \sum_{j=0}^{N}(u_{i}- u_{j}) \Delta x \Big \vert ^2 \Delta x \leq \sum_{i=0}^{N} \sum_{j=0}^{N} (u_{i}-u_{j})^2 \Delta x \Delta x \label{zizou}
\end{align} 
where the last inequality is obtained thanks to Jensen's inequality.
Then for $i,j \in \lbrace 0,\dots,N \rbrace$ we have, using a Cauchy-Schwarz inequality, that if $i > j$
\begin{align*}
\Big \vert \frac{u_{i}-u_{j}}{\Delta x} \Big \vert^2 =  \Big \vert \sum_{k=j}^{i-1} \frac{u_{k+1}-u_{k}}{\Delta x} \Big \vert ^2 \leq (i-j) \sum_{k=j}^{i-1} \Big \vert \frac{u_{k+1}-u_{k}}{\Delta x} \Big \vert^2 \leq (i-j) \sum_{k=0}^{N} \Big \vert \frac{u_{k+1}-u_{k}}{\Delta x} \Big \vert^2.
\end{align*}
while if $i \leq j$ 
\begin{align*}
\Big \vert \frac{u_{j}-u_{i}}{\Delta x} \Big \vert^2 =  \Big \vert \sum_{k=i}^{j-1} \frac{u_{k+1}-u_{k}}{\Delta x} \Big \vert ^2 \leq (j-i) \sum_{k=i}^{j-1} \Big \vert \frac{u_{k+1}-u_{k}}{\Delta x} \Big \vert^2 \leq (j-i) \sum_{k=0}^{N} \Big \vert \frac{u_{k+1}-u_{k}}{\Delta x} \Big \vert^2.
\end{align*}
Plugging this inequality in \eqref{zizou} we obtain
\begin{align*}
    \sum_{i=0}^{N} \Big \vert u_{i} -\langle u \rangle \Big \vert^2 \Delta x \leq  \Big( \sum_{k=0}^{N} \Big \vert \frac{u_{k+1}-u_{k}}{\Delta x}  \Big \vert^2  \Delta x \Big) \sum_{i=0}^{N} \sum_{j=0}^{N} \big \vert (i-j)\Delta x  \big \vert   \Delta x \Delta x.
\end{align*}
Then a direct computation yields 
\begin{align*}
    \displaystyle \sum_{i=0}^{N} \sum_{j=0}^{N} \vert i-j \vert \Delta x^3 = \sum_{i=0}^{N} i(i+1) \Delta x^{3} = \frac{1}{3} \frac{N(N+1)(N+2)}{(N+1)^3} = \frac{1}{3} \left ( 1 - \frac{2}{(N+1)^{3}} \right)
\end{align*}
which enables to conclude the proof.
\end{proof}
\end{lemma}
We are ready to prove Theorem \ref{main_result_1}.
\begin{proof} Consider the assumption of Theorem \ref{main_result_1}. The existence part is done by an induction argument. So we assume that for a fixed integer $n \in \lbrace 0,\dots,N_{T}-1 \rbrace$ we have been able to construct a solution $(\rho^{n},u^{n},\phi^{n}) \in X(\mathcal{T}) \times X(\mathcal{T}^{\star}) \times X(\mathcal{T})$ with $\rho^{n} > 0.$ We want to prove the existence of a solution at step $n+1$. 
We shall apply the Brouwer fixed-point theorem. Consider the ball of radius $M$ centered at $u^{n}$
$$
B_{M} := \Big\lbrace u \in X(\mathcal{T}^{*}) \: : \: \| u - u^{n} \|_{L^{2}(\T)} \leq M \Big\rbrace
$$
where $M > 0$ is to be fixed later.
Note that $B_{M}$ is closed for the $L^{2}$-topology and that $X(\mathcal{T}^{\star})$ is a finite dimensional space. 
\paragraph{Definition of a Mapping.}
We consider 
$T : B_{M} \longrightarrow X(\mathcal{T}^{\star})$ which is defined on $B_{M}$ in three steps. For $u \in B_{M}:$
\begin{itemize}
\item{Step 1:} Compute $ \overline{\rho}(u)$ that solves for $i \in \lbrace 0,..,N \rbrace:$
\begin{equation} \label{def:barho}
\frac{\overline{\rho}_{i} - \rho_{i}^{n}}{\Delta t} + \frac{ \mathcal{F}_{i+\frac{1}{2}}(u) - \mathcal{F}_{i-\frac{1}{2}}(u)}{\Delta x} = 0
\end{equation}
where 
\begin{equation} \label{ffflux}
    \mathcal{F}_{i+\frac{1}{2}}(u) = G(\overline{\rho}_{i},\overline{\rho}_{i+1},u_{i+\frac{1}{2}}).
\end{equation}
Since the flux \eqref{def:G} is linear in its two first arguments, the equation on $\overline{\rho}$ can be written under the form of a linear system $L(u) \overline{\rho} = \rho^{n}$ where $L$ is a M-matrix of size $N+1$ given for $i \in \lbrace 0,\dots,N \rbrace$ by:
\begin{align*}
&L_{i,i} = 1 + \frac{\Delta t}{\Delta x} \Big( g(u_{i+\frac{1}{2}}) +g(u_{i-\frac{1}{2}}) - u_{i-\frac{1}{2}} \Big),\\
&L_{i,i+1} = -\frac{\Delta t}{\Delta x} (g(u_{i+\frac{1}{2}}) - u_{i+\frac{1}{2}})\mathbf{1}_{i+1 \leq N},\\
&L_{i,i-1} = -\frac{\Delta t}{\Delta x}g(u_{i-\frac{1}{2}})\mathbf{1}_{i-1 \geq 0}.
\end{align*}
So $\overline{\rho}$ is uniquely defined and since $\rho^{n} > 0$ we have $\overline{\rho}  > 0$. The fact $L$ is a M-matrix comes from $L$ has positive diagonal terms ($L_{i,i}>0$) and non-positive off-diagonal terms ($L_{i,j}\leq 0$ for $i\neq j$) and is strictly diagonally dominant with respect to their columns. 
For the latter argument, we indeed have 
$L_{i,i} > \displaystyle \sum_{j\neq i} |L_{j,i}|=|L_{i-1,i}|+|L_{i+1,i}|$ since 
\begin{align*}
1+\frac{\Delta t}{\Delta x}\Big(g(u_{i+\frac{1}{2}}) +g(u_{i-\frac{1}{2}}) - u_{i-\frac{1}{2}} \Big) &> \frac{\Delta t}{\Delta x}\Big( g(u_{i-\frac{1}{2}})- u_{i-\frac{1}{2}}+ g(u_{i+\frac{1}{2}})\Big).
\end{align*} 
See \cite{claire, Gastaldo}. 
\item{Step 2:} Compute $\varphi = \varphi(\bar{\rho}(u))$ which solves the non linear discrete Poisson equation for $i \in  \lbrace 0,..,N\rbrace:$
$$
\varepsilon^2 (\Delta \varphi)_{i} + e^{-\varphi_i} = \overline{\rho}_i.
$$
Existence and uniqueness for this non linear equation is classical and can be proven for example by minimization of a strictly convex functional. Mimicking exactly the computation to obtain the estimate \eqref{nedved}, we get using a Cauchy-Schwarz inequality combined with the Poincaré-Wirtinger inequality \eqref{discrete_PW_ineq},
\begin{align*} 
    | \varphi |_{H^{1}(\T)} \leq \frac{1}{ \varepsilon^2 \sqrt{3}}  \| \bar{\rho} - e^{\langle \varphi \rangle} \|_{L^{2}(\T)},  
\end{align*}
and using \eqref{ineg_phi-avphi}, we obtain 
\begin{align} \label{nico}
    | \varphi |_{H^{1}(\T)} \leq \frac{1}{ \varepsilon^2 \sqrt{3}}  \| \bar{\rho} - e^{\langle \varphi \rangle} \|_{L^{2}(\T)} \leq \frac{1}{\varepsilon^2\sqrt{3}}  \left( \|\bar{\rho}\|_{L^{2}(\T)} +\| \bar{\rho}\|_{L^{1}(\T)} \right) \leq C\Big(\Delta x, \|\rho^{n}\|_{L^{1}(\T)}, \frac{1}{\varepsilon^2}\Big), 
\end{align}
where  $C(\Delta x, \|\rho^{n}\|_{L^{1}(\T))},\frac{1}{\varepsilon^2}) > 0$ is constant that depends only on $\Delta x$ and $\| \rho^{n} \|_{L^{1}(\T)}$. This constant is obtained thanks to the equivalence of norms in finite dimension, the positivity of $\bar{\rho}$ and the conservation of mass given by \eqref{def:barho}.

\item{Step 3:} Compute $v = v(\bar{\rho}(u),\varphi(\bar{\rho}(u)),u)$ which solves for $i \in \lbrace 0,...,N\rbrace:$
\begin{align} 
\label{taratatu}
\frac{ \overline{\rho}_{i+\frac{1}{2}} v_{i+\frac{1}{2}} - \rho_{i+\frac{1}{2}}^{n} u_{i+\frac{1}{2}}^{n} }{\Delta t} + \frac{{\cal Q}_{i+1}(u) v_{i+1}-{\cal Q}_{i}(u) v_{i}}{\Delta x} = \bar{\tilde{\rho}}_{i+\frac{1}{2}} \delta (\varphi)_{i+\frac{1}{2}},
\end{align}
where
\begin{align}
\overline{\rho}_{i+\frac{1}{2}} = \frac{ \overline{\rho}_{i} + \overline{\rho}_{i+1}}{2}, \quad Q_{i}(u) = \frac{ \mathcal{F}_{i+\frac{1}{2}}(u) +\mathcal{F}_{i-\frac{1}{2}}(u) }{2}, 
\quad v_{i} = \begin{cases}
        v_{i-\frac{1}{2}} \textnormal{ if } \mathcal{Q}_{i}(u) \geq 0,\\
        v_{i+\frac{1}{2}} \textnormal{ if }\mathcal{Q}_{i}(u) < 0
    \end{cases}
\end{align}
and
\begin{equation} \label{tilde_rrro}
\displaystyle \bar{\tilde{\rho}}_{i+\frac{1}{2}} =
\begin{cases}
\frac{ G(\overline{\rho}_{i},\overline{\rho}_{i+1},u_{i+\frac{1}{2}}) - G(\overline{\rho}_{i},\overline{\rho}_{i+1},0) }{ u_{i+\frac{1}{2}} }  \;\; \textnormal{ if } u_{i+\frac{1}{2}} \neq 0,\\
        \bar{\rho}_{i+1} -(\bar{\rho}_{i+1}-\bar{\rho}_{i})g'(0) \;\;\;\;\;\;\;\textnormal{ if } u_{i+\frac{1}{2}} = 0.
    \end{cases}
\end{equation}
Note that $v$ solves a linear system which is invertible.
\end{itemize}
Step 1, Step 2 and Step 3 are well defined, so is $T$ on $B_{M}.$ $T$ is moreover continuous on $B_{M}$ notably because the flux given in \eqref{ffflux} is continuous, the flux part in \eqref{taratata} is also continuous and the forcing term in \eqref{taratata} has been designed in such a way that \eqref{tilde_rrro} is in particular continuous with respect to $u.$
\paragraph{Stability of $B_{M}.$}
We want to prove that $T(B_{M}) \subset B_{M}$ for a well-chosen $M > 0$.
We first mimick the energy estimates as in Proposition \ref{energy_decay}. Note that because the right hand side in \eqref{taratatu} is now explicit, we do not have the energy decay. We have instead, after multiplying \eqref{taratatu} by $v_{i+1/2}$ 
\begin{align}\label{bangkok}
    \sum_{i=0}^{N} \rho_{i+\frac{1}{2}}^{n} \frac{(v_{i+\frac{1}{2}} - u_{i+\frac{1}{2}}^{n})^2}{2} \Delta x +  \mathcal{H}(\bar{\rho},v,\varphi) \leq \mathcal{H}(\rho^{n},u^{n},\phi^{n}) + \Delta t \sum_{i=0}^{N} \bar{\tilde{\rho}}_{i+\frac{1}{2}} (v_{i+\frac{1}{2}}-u_{i+\frac{1}{2}}) \delta( \varphi)_{i+\frac{1}{2}} \Delta x.
\end{align}
The first term comes from the implictness of the discrete time derivative (its equivalent, is the third term in \eqref{def:S}).
We need to estimate the residual term
$P = \Delta t\sum_{i=0}^{N} \bar{\tilde{\rho}}_{i+\frac{1}{2}} (v_{i+\frac{1}{2}}-u_{i+\frac{1}{2}}) (\delta \phi)_{i+\frac{1}{2}} \Delta x. $ Using the definition of $G$ given in \eqref{def:G} and the fact that $g$ is Lipschitz continuous, we have for $i \in \lbrace 0,...,N \rbrace:$
\begin{align}
    | \bar{\tilde{\rho}}_{i+\frac{1}{2}} | \leq (1 +2\textnormal{Lip}(g))\| \bar \rho \|_{L^{\infty}(\T)} \leq \frac{(1+2\textnormal{Lip}(g))}{\Delta x} \| \bar{\rho} \|_{L^{1}(\T)} \leq \frac{(1+2\textnormal{Lip}(g))}{\Delta x} \| \rho^{n} \|_{L^{1}(\T)},
\end{align}
where we used the equivalence of norms  in finite dimension, the conservation of the positivity and the total mass given by \eqref{def:barho}. Using a Hölder inequality and the elliptic estimate \eqref{nico} we obtain
\begin{align}
    |P| \leq \Delta t C'\Big( \Delta x, \textnormal{Lip}(g),\|\rho^{n}\|_{L^{1}(\T)}, \frac{1}{\varepsilon^2} \Big) \| v - u \|_{L^{2}(\T)}
\end{align}
where $C'\Big( \Delta x, \textnormal{Lip}(g),\|\rho^{n}\|_{L^{1}(\T)}, \frac{1}{\varepsilon^2} \Big) > 0$ is a constant that depends only on $\Delta x$, $\varepsilon$, $\textnormal{Lip}(g)$ and $\| \rho^{n} \|_{L^{1}(\T)}.$ We set $C^{n} :=C'\Big( \Delta x, \textnormal{Lip}(g),\|\rho^{n}\|_{L^{1}(\T)}, \frac{1}{\varepsilon^2} \Big)$ in the remaining part of the proof. Using a triangular inequality and the fact that $u \in B_{M}$, we have the following estimate for $P$,
\begin{align} \label{nicosie}
    | P | \leq \Delta t C^{n} \| v-u^{n}\|_{L^{2}(\T)} + \Delta t C^{n} M.
\end{align}
Besides, we observe that the discrete energy functional is bounded below,
\begin{align} \label{below_h}
    \mathcal{H}(\bar{\rho},v,\varphi) \geq \bar{h}, \;\;\; \mbox{ where } \;\;\; \bar{h} = \underset{ \R}\min \, h > -\infty. 
\end{align}
 Combining \eqref{nicosie} and \eqref{below_h} with \eqref{bangkok}, we obtain that
\begin{align} \label{piste_rouge}
    \frac{\underset{ i \in \lbrace 0,..,N \rbrace}\min \rho^{n}_{i+\frac{1}{2}}}{2} \| v -u^{n} \|_{L^{2}(\T)}^2 - \Delta t C^{n} \| v -u^{n} \|_{L^{2}(\T)} - (\Delta t C^{n} M + \mathcal{H}(\rho^{n},u^{n},\phi^{n}) - \bar{h}) \leq 0,
\end{align}
where we recall that $\mathcal{H}(\rho^{n},u^{n},\phi^{n}) -\bar{h} \geq 0$ since $\sum_{i = 0}^{N} (h(\phi_{i}^{n}) - \bar{h} )\Delta x \geq 0$ by definition of $\bar{h}.$ Note that \eqref{piste_rouge} is a polynomial of second degree in $\| v -u^{n} \|$ and the inequality \eqref{piste_rouge} tells us that this polynomial is non positive on $\R^{+}.$ It has two roots of opposite sign and the non negative root is given by
$$
X = \frac{\Delta t C^{n} + \sqrt{ (\Delta t C^{n})^2 + 2( \Delta t C^{n} M + \mathcal{H}(\rho^{n},u^{n},\phi^{n}) - \bar{h})\underset{i \in \lbrace 0,..,N \rbrace}\min \rho^{n}_{i+\frac{1}{2}}}}{\underset{ i \in \lbrace 0,...,N \rbrace}\min \rho^{n}_{i+\frac{1}{2}}}. 
$$
Note that that the denominator is fixed while the numerator behaves as $O(\sqrt{M})$ as $M \rightarrow +\infty$. 
Therefore, we claim that there exists $M > 0$ large enough, which possibly depends on $(\rho^{n}, u^{n},\phi^{n}, \Delta x, \Delta t, \varepsilon)$ such that $X \leq M$. Thus  \eqref{piste_rouge} implies that $\| v-u^{n} \|_{L^{2}(\T)} \leq X \leq M.$ We have proven that $T(B_{M}) \subset B_{M}.$ The Brouwer-fixed point theorem thus applies. By induction we deduce the existence of a solution for $n \in \lbrace 0,\dots,N_{T} \rbrace.$ The estimate \eqref{positivity} is a consequence of the positivity of the scheme. The estimate \eqref{conservativity} is obtained by summation of the discrete continuity equation and the positivity of $\rho^{n}$. The energy decay \eqref{energy_estimates} is just a direct application of Proposition \ref{energy_decay}. Finally the discrete elliptic estimates are a consequence of Proposition \ref{discrete_elliptic_estimates}.
\end{proof}

\subsection{ Discrete modulated energy estimate}\label{sec:discrete_modulated_energy}
This section is devoted to the proof of Theorem \ref{main_result_2}.
The first step consists in establishing the discrete evolution law for the modulated energy \eqref{modulated_energy_discrete}.
\begin{lemma}{(Evolution of the discrete modulated energy)} Let $(\rho^{n},u^{n},\phi^{n})_{n = 0,\dots, N_{T}} \subset X(\mathcal{T}) \times X(\mathcal{T}^{\star}) \times X(\mathcal{T}) $ a solution to \eqref{EPBD} such that $\rho^{n} > 0$. Assume $\phi^{0} \in X(\mathcal{T})$ verifies the discrete Poisson equation initially. Let $(\bar{u},\bar{\phi}) \in X(\mathcal{T}^{\star}) \times X(\mathcal{T})$ a constant state. Then the solution satisfies for $0 \leq n \leq N_{T}-1,$
\begin{align} \label{mod_energy_n+1}
    {\cal E}(\rho^{n+1},u^{n+1},\phi^{n+1} | \bar{u}, \bar{\phi}) = {\cal E}(\rho^{n},u^{n},\phi^{n} |  \bar{u}, \bar{\phi})  - \Delta t \tau^{n+1} - \Delta t \bar{u} \sum_{i=0}^{N}\tilde{\rho}^{n+1}_{i+\frac{1}{2}} \delta (\phi^{n+1})_{i+\frac{1}{2}} \Delta x,
\end{align}
where we recall that $\tau^{n+1} \geq 0$ is given by \eqref{tau_n}.
\begin{proof} Let $0 \leq n \leq N_{T}-1$. We set $\bar{\rho} = e^{-\bar{\phi}}.$ Thanks to the decomposition \eqref{discrete_decomp_E} we have
\begin{align*}
    & {\cal E}(\rho^{n+1},u^{n+1},\phi|\bar{u}, \bar{\phi}) = \mathcal{H}(\rho^{n+1},u^{n+1},\phi^{n+1}) + {\cal E}_{kin}(\rho^{n+1}, u^{n+1}|\bar{u}) - {\cal E}_{int}(\phi^{n+1}|\bar{\phi})\\
     & = \mathcal{H}(\rho^{n},u^{n},\phi^{n}) - \Delta t \tau^{n+1} + {\cal E}_{kin}(\rho^{n+1}, u^{n+1}|\bar{u}) - {\cal E}_{int}(\phi^{n+1}|\bar{\phi}), 
\end{align*}
where we used the energy decay \eqref{energy_estimates}.
Besides, a direct computation gives
\begin{align*}
{\cal E}_{int}(\phi^{n+1} | \bar{\phi}) = \sum_{i=0}^{N} \Big( \tilde{h}(\bar{\rho})+\tilde{h}'(\bar{\rho})(e^{-\phi_{i}^{n+1}}-\bar{\rho})  \Big) \Delta x. 
\end{align*}
Recall that $\bar{\rho}$ is constant. So, using the discrete Poisson  we have $e^{-\phi_{i}^{n+1}} - \bar{\rho} = \rho_{i}^{n+1} - \bar{\rho} - \varepsilon^2 \Delta(\phi^{n+1})_{i}.$ 
So with the periodicity, 
\begin{align*}
{\cal E}_{int}(\phi^{n+1} | \bar{\phi}) = \sum_{i=0}^{N} \Big( \tilde{h}(\bar{\rho})+\tilde{h}'(\bar{\rho})(\rho_{i}^{n+1}-\bar{\rho})  \Big) \Delta x = \sum_{i=0}^{N} \Big( \tilde{h}(\bar{\rho})+\tilde{h}'(\bar{\rho})(\rho_{i}^{n}-\bar{\rho})  \Big) \Delta x = {\cal E}_{int}(\phi^{n} | \bar{\phi}),
\end{align*}
where the last equality is obtained thanks to the mass conservation. We compute the evolution of the  modulated kinetic energy. A direct computation yields
\begin{align*}
    \frac{{\cal E}_{kin}(\rho^{n+1},u^{n+1} | \bar{u}) -  {\cal E}_{kin}(\rho^{n},u^{n} | \bar{u})}{\Delta t} = \sum_{i=0}^{N}  \frac{\bar{u}^2}{2} \Big( \frac{\rho_{i+\frac{1}{2}}^{n+1} - \rho_{i+\frac{1}{2}}^{n}}{\Delta t} \Big) \Delta x - \sum_{i=0}^{N} \bar{u} \Big( \frac{\rho_{i+\frac{1}{2}}^{n+1}u_{i+\frac{1}{2}}^{n+1} - \rho_{i+\frac{1}{2}}^{n} u_{i+\frac{1}{2}}^{n}}{\Delta t} \Big) \Delta x.
\end{align*}
Using the discrete continuity equation and the discrete momentum we obtain
\begin{align*}
 \frac{{\cal E}_{kin}(\rho^{n+1},u^{n+1} | \bar{u}) -  {\cal E}_{kin}(\rho^{n},u^{n} | \bar{u})}{\Delta t} = - \bar{u} \sum_{i=0}^{N} \tilde{\rho}^{n+1}_{i+\frac{1}{2}} \delta( \phi^{n+1})_{i+\frac{1}{2}} \Delta x.
\end{align*}
Then gathering all the terms together, we obtain the expected equality.
\end{proof}
\end{lemma}
To obtain a close estimate for the equation \eqref{mod_energy_n+1},  we need to control the term
\begin{align} \label{work_of_force}
    W^{n+1} = \sum_{i=0}^{N} \tilde{\rho}^{n+1}_{i+\frac{1}{2}} \delta (\phi^{n+1})_{i+\frac{1}{2}} \Delta x.
\end{align}
In this regard we have.
\begin{lemma}{(Control on W)} \label{control_W}
Let $(\rho^{n},u^{n},\phi^{n})_{n = 0,\dots, N_{T}} \subset X(\mathcal{T}) \times X(\mathcal{T}^{\star}) \times X(\mathcal{T}) $ a solution to \eqref{EPBD} such that $\rho^{n} > 0$. Assume $\phi^{0} \in X(\mathcal{T})$ verifies the discrete Poisson equation initially. Then we have for $n \in \lbrace 0,\dots,N_{T}-1 \rbrace,$
\begin{align} \label{decomp_W}
    &W^{n+1} = W_{1}^{n+1} + W_{2}^{n+1},\\
    &W_{1}^{n+1}  = - \sum_{i=0}^{N}(\rho_{i+1}^{n+1}-\rho_{i}^{n+1}) \hat{g}(u_{i+\frac{1}{2}}^{n+1}) \delta (\phi^{n+1})_{i+\frac{1}{2}} \Delta x ,\\
    &W_{2}^{n+1} =\sum_{i=0}^{N} \rho_{i+1}^{n+1} \delta (\phi^{n+1})_{i+\frac{1}{2}}  \Delta x,
\end{align}
where $\hat{g}(u) = \begin{cases}
   \displaystyle \frac{g(u) - g(0)}{u} \; \mbox{ if } \; u \neq 0, \\
    g'(0) \quad\quad\quad \,\, \mbox{ if } \; u = 0.
\end{cases}$ 
Moreover, we have the following estimates 
\begin{align} \label{estimate_W1}
    | W_{1}^{n+1}| \leq \frac{8 \textnormal{Lip}(g)}{\Delta x} {\cal E}(\rho^{n+1},u^{n+1},\phi^{n+1} | \bar{u}, \bar{\phi}) + \frac{\textnormal{Lip(g)}}{\Delta x} \sum_{i=0}^{N} \Big |( \phi_{i+1}^{n} - \phi_{i}^{n}) (e^{\phi_{i+1}^{n}} - e^{-\phi_{i}^{n}} ) \Big | \Delta x,
\end{align}
\begin{align}\label{estimate_W2}
    \big | W_{2}^{n+1} \big | \leq \frac{4}{\Delta x} {\cal E}\big( \rho^{n+1},u^{n+1},\phi^{n+1}| \bar{u},\bar{\phi}\big) + \frac{2\Delta x}{\varepsilon^2} \| \rho^{n+1} \|_{L^{\infty}(\T)} {\cal E}(\rho^{n+1},u^{n+1},\phi^{n+1} | \bar{u},\bar{\phi}). 
\end{align}

\begin{proof} Let $0 \leq n \leq N_{T}-1.$
Using the definitions  \eqref{rho_upwind}, \eqref{def:G} and the definition of $\hat{g}$, $W^{n+1}$ can be written as 
\begin{align*}
W^{n+1} &= - \sum_{i=0}^{N}(\rho_{i+1}^{n+1}-\rho_{i}^{n+1}) \hat{g}(u_{i+\frac{1}{2}}^{n+1}) \delta (\phi^{n+1})_{i+\frac{1}{2}} \Delta x + \sum_{i=0}^{N} \rho_{i+1}^{n+1} \delta( \phi^{n+1})_{i+\frac{1}{2}}  \Delta x\\
\end{align*}

Let observe that $\hat{g}$ 
is a continuous and bounded function with $\| \hat{g} \|_{L^{\infty}(\R)} \leq \textnormal{Lip}(g).$ It shows the decomposition \eqref{decomp_W}.
We now estimate each term separately.
Using the discrete Poisson equation, we decompose the first term as
\begin{align*}
W_{1}^{n+1} = W_{1,1}^{n+1} + W_{1,2}^{n+1}, \\
W_{1,1}^{n+1} = - \sum_{i=0}^{N} \varepsilon^2 \left( \Delta(\phi^{n+1})_{i+1} - \Delta (\phi^{n+1})_{i} \right) \delta( \phi^{n+1})_{i+\frac{1}{2}} \hat{g}(u_{i+\frac{1}{2}}^{n+1}) \Delta x,\\
W_{1,2}^{n+1} = - \sum_{i=0}^{N} \left( e^{-\phi_{i+1}^{n+1}} - e^{-\phi_{i}^{n+1}} \right) \delta (\phi^{n+1})_{i+\frac{1}{2}} \hat{g}(u_{i+\frac{1}{2}}^{n+1}) \Delta x.
\end{align*}
To estimate $W_{1,1}^{n+1}$, we use a discrete integration by parts then we get
\begin{align*}
&W_{1,1}^{n+1} = \sum_{i=0}^{N} \varepsilon^2 \Delta (\phi^{n+1})_{i} \left( \hat{g}(u_{i+\frac{1}{2}}^{n+1}) \delta (\phi^{n+1})_{i+\frac{1}{2}} -\hat{g}(u_{i-\frac{1}{2}}^{n+1}) \delta (\phi^{n+1})_{i-\frac{1}{2}} \right) \Delta x \\
&= \frac{1}{\Delta x} \sum_{i=0}^{N} \varepsilon^2 \big( \delta (\phi^{n+1})_{i+\frac{1}{2}} - \delta (\phi^{n+1})_{i-\frac{1}{2}} \big)  \left( \hat{g}(u_{i+\frac{1}{2}}^{n+1}) \delta (\phi^{n+1})_{i+\frac{1}{2}} -\hat{g}(u_{i-\frac{1}{2}}^{n+1}) \delta (\phi^{n+1})_{i-\frac{1}{2}} \right) \Delta x.
\end{align*}
Expanding the product, using a Young inequality and a translation of indices, we obtain
\begin{align} \label{estimate_W_1_1}
    | W_{1,1}^{n+1}| \leq \frac{8 \textnormal{Lip}(g)}{\Delta x} \sum_{i=0}^{N} \frac{\varepsilon^2}{2} \big \vert  \delta (\phi^{n+1})_{i+\frac{1}{2}}|^2 \Delta x \leq \frac{8 \textnormal{Lip}(g)}{\Delta x} {\cal E}(\rho^{n+1},\phi^{n+1},u^{n+1} |  \bar{u}, \bar{\phi}).
\end{align}
As for $W_{1,2}^{n+1},$ we obtain readily
\begin{align*}
| W_{1,2}^{n+1} | \leq \frac{ \textnormal{Lip(g)}}{\Delta x} \sum_{i=0}^{N} \Big | (\phi_{i+1}^{n} - \phi_{i}^{n}) (e^{\phi_{i+1}^{n}} - e^{-\phi_{i}^{n}} ) \Big | \Delta x.
\end{align*}
We thus infer the estimate \eqref{estimate_W1}.

We now treat the term $W_2^{n+1}$. We first use the discrete Poisson equation to get
\begin{align*}
    W_{2}^{n+1} = \sum_{i=0}^{N} \varepsilon^2 \Delta (\phi^{n+1})_{i+1} \delta (\phi^{n+1})_{i+\frac{1}{2}} \Delta x + \sum_{i=0}^{N} e^{-\phi_{i+1}^{n+1}} \delta (\phi^{n+1})_{i+\frac{1}{2}}  \Delta x.
\end{align*}
We set
\begin{align*}
    &W_{2,1}^{n+1} =  \sum_{i=0}^{N} \varepsilon^2 \Delta (\phi^{n+1})_{i+1} \delta (\phi^{n+1})_{i+\frac{1}{2}} \Delta x \;\; \mbox{ and } \;\; W_{2,2}^{n+1} = \sum_{i=0}^{N} e^{-\phi_{i+1}^{n+1}} \delta (\phi^{n+1})_{i+\frac{1}{2}}  \Delta x.
\end{align*}
We are going to bound each term separately. We have,
\begin{align*}
W_{2,1}^{n+1}  = \sum_{i=0}^{N} \frac{\varepsilon^2}{\Delta x} \left( {\delta( \phi^{n+1})_{i+\frac{3}{2}} - \delta (\phi^{n+1})_{i+\frac{1}{2}}} \right) \delta (\phi^{n+1})_{i+\frac{1}{2}} \Delta x.
\end{align*}
Using a Young inequality and a translation of indices we obtain the following bound for $W_{2,1}^{n+1},$
\begin{align*}
    \big |  W_{2,1}^{n+1} \big | \leq \frac{2}{\Delta x} \sum_{i=0}^{N} \varepsilon^2 \vert \delta (\phi^{n+1})_{i+\frac{1}{2}}\vert^2 \Delta x \leq \frac{4}{\Delta x} {\cal E}\big( \rho^{n+1},\phi^{n+1},u^{n+1}| \bar{u},\bar{\phi}\big).
\end{align*}
As for the  term $W_{2,2}^{n+1}$, we use a Taylor-Lagrange expansion: for each $i \in \lbrace 0,\dots,N\rbrace$ there exists $\zeta_{i}^{n+1} \in \big( \min(\phi_{i+1}^{n+1},\phi_{i}^{n+1}),\max(\phi_{i+1}^{n+1},\phi_{i}^{n+1}) \big)$
\begin{align*}
    e^{-\phi_{i}^{n+1}} = e^{-\phi_{i+1}^{n+1}}  - e^{-\phi_{i+1}^{n+1}}(\phi_{i+1}^{n+1}-\phi_{i}^{n+1}) + e^{-\zeta_{i}^{n+1}}\frac{(\phi_{i+1}^{n+1}-\phi_{i}^{n+1})^2}{2}.
\end{align*}
Using the periodicity, we obtain after summation that 
\begin{align*}
W_{2,2}^{n+1} = -\frac{1}{\Delta x} \sum_{i=0}^{N}  e^{-\zeta_{i}^{n+1}}\frac{(\phi_{i+1}^{n+1}-\phi_{i}^{n+1})^2}{2} \Delta x = - \Delta x \sum_{i=0}^{N} e^{-\zeta_{i}^{n+1}} |\delta (\phi^{n+1})_{i+\frac{1}{2}}|^2 \Delta x.
\end{align*}
Using the maximum principle \eqref{maximum_principle} we obtain the bound $e^{-\xi_{i}^{n+1}} \leq \| \rho^{n+1}\|_{L^{\infty}(\T)}$ for $i \in \lbrace 0,\dots,N\rbrace$, so that 
we deduce
\begin{align*}
    \big| W_{2,2}^{n+1} \big| \leq \frac{2\Delta x}{\varepsilon^2} \| \rho^{n+1} \|_{L^{\infty}(\T)} {\cal E}(\rho^{n+1},\phi^{n+1},u^{n+1} | \bar{u},\bar{\phi}).
\end{align*}
So that eventually,
\begin{align*}
\big | W_{2}^{n+1} \big | \leq \frac{4}{\Delta x} {\cal E}\big( \rho^{n+1},\phi^{n+1},u^{n+1}| \bar{u},\bar{\phi}\big) + \frac{2\Delta x}{\varepsilon^2} \| \rho^{n+1} \|_{L^{\infty}(\T)} {\cal E}(\rho^{n+1},\phi^{n+1},u^{n+1} | \bar{u},\bar{\phi})
\end{align*}
which is the expected estimate \eqref{estimate_W2}
\end{proof}  
\end{lemma}
We are ready to prove Theorem \ref{main_result_2}.
\begin{proof}
Combining the estimate of Lemma \ref{control_W}, with the evolution of the modulated energy  \eqref{mod_energy_n+1}, we obtain the following following closed estimate for the modulated energy: 

\begin{align}
    &\forall n \in \lbrace 0,\dots,N_{T}-1 \rbrace, \quad {\cal E}(\rho^{n+1},u^{n+1},\phi^{n+1} | \bar{u},\bar{\phi}) \leq {\cal E}(\rho^{n},u^{n},\phi^{n}| \bar{u},\bar{\phi}) \label{estimate_E_n+1} \\
    &+ | \bar{u} | \frac{\Delta t}{\Delta x} \Big( 8  \textnormal{Lip}(g) + 4  + \frac{2 \Delta x^2}{\varepsilon^2} \| \rho^{n+1} \|_{L^{\infty}(\T)} \Big){\cal E}(\rho^{n+1},u^{n+1},\phi^{n+1} | \bar{u}, \bar{\phi})\nonumber \\
    &- \Delta t \tau^{n+1} + \frac{\Delta t |\bar{u}|\textnormal{ Lip(g)} }{\Delta x} \sum_{i=0}^{N} \Big | (\phi_{i+1}^{n} - \phi_{i}^{n}) (e^{\phi_{i+1}^{n}} - e^{-\phi_{i}^{n}} ) \Big | \Delta x.\nonumber
\end{align}
Now remark that in $\Delta t \tau^{n+1}$ \eqref{tau_n} we have the term
$$
\frac{\Delta t g(0)}{\Delta x} \sum_{i=0}^{N} \big \vert (e^{-\phi_{i}^{n+1}} - e^{-\phi_{i-1}^{n+1}})(\phi_{i}^{n+1}-\phi_{i-1}^{n+1}) \big \vert \Delta x.
$$
So we see that the last term in \eqref{estimate_E_n+1} can be absorbed by $-\Delta t \tau^{n+1}$ if $\vert \bar{u} \vert \leq \frac{g(0)}{\textnormal{Lip}(g)}.$ The case $\bar{u} = 0$ is trivial. It completes the proof of Theorem \ref{main_result_2}.
\end{proof}

\section{ Numerical experiment}\label{sec:numerical_experiment}
This section is dedicated to some numerical experiments. First, we describe how we solve the non linear scheme. Then, a test illustrating the theorems proven above (namely Theorem \ref{main_result_1} and Theorem \ref{main_result_2}) is performed and we explore the accuracy of our scheme. Finally, a test taken from \cite{fiveb, degond_b} is considered. 

\subsection{Description of the non linear solver}
Since we shall use a Newton method, we assume that the function $g$ verifies, in addition to \eqref{ass_1_g}-\eqref{ass_3_g},  $g \in \mathrm{C}^{1}(\R)$ and that it is twice differentiable near the origin.
Suppose now that for a fixed integer $n \in \lbrace 0,\dots,N_{T}-1\rbrace$ we have constructed the solution $(\rho^{n},u^{n},\phi^{n})\in X(\mathcal{T}) \times X(\mathcal{T}^{\star}) \times X(\mathcal{T})$. We seek a solution of the scheme \eqref{EPBD} at step $n+1$ as a fixed point of a certain map just exactly as in the proof of existence.
More precisely, we look for $u^{n+1} \in X(\mathcal{T}^{\star})$ which solves
$$
\textbf{T}^{n}(u^{n+1}) = u^{n+1}
$$
where $\textbf{T}^{n} : X(\mathcal{T}^{\star}) \longrightarrow X(\mathcal{T}^{*})$
is defined for $u \in X(\mathcal{T}^{\star})$ in three steps:
\begin{itemize}
    \item{ Step 1: we solve the linear continuity equation of unknown $\bar{\rho} \equiv \bar{\rho}(u)$ which solves \eqref{def:barho}.}
    \item{ Step 2: we solve the non linear Poisson equation of unknown $\varphi \equiv\varphi(\bar{\rho}) \in X(\mathcal{T})$ which  verifies for $i \in \lbrace 0,...,N\rbrace:$}
    \begin{equation*}
    \varepsilon^2 (\Delta \varphi)_{i} + e^{-\varphi_{i}} = \bar{\rho}_{i}.
    \end{equation*}
     We use a  Newton-method with an error bound fixed to the zero machine which is in our case $10^{-15}.$
     \item{ Step 3: we solve the non linear momentum equation for $v \equiv v(\bar{\rho},u) \in X(\mathcal{T}^{\star})$ which verifies for $i \in \lbrace 0,...,N\rbrace:$}
     \begin{align} 
\label{taratata}
\frac{ \overline{\rho}_{i+\frac{1}{2}} v_{i+\frac{1}{2}} - \rho_{i+\frac{1}{2}}^{n} u_{i+\frac{1}{2}}^{n} }{\Delta t} + \frac{{\cal Q}_{i+1}(u) v_{i+1}-{\cal Q}_{i}(u) v_{i}}{\Delta x} = \bar{\tilde{\rho}}(v)_{i+\frac{1}{2}} (\delta \varphi)_{i+\frac{1}{2}},
\end{align}
where
\begin{align}
\overline{\rho}_{i+\frac{1}{2}} = \frac{ \overline{\rho}_{i} + \overline{\rho}_{i+1}}{2}, \quad Q_{i}(u) = \frac{ \mathcal{F}_{i+\frac{1}{2}}(u) +\mathcal{F}_{i-\frac{1}{2}}(u) }{2}, 
\quad v_{i} = \begin{cases}
        v_{i-\frac{1}{2}} \textnormal{ if } \mathcal{Q}_{i}(u) \geq 0,\\
        v_{i+\frac{1}{2}} \textnormal{ if }\mathcal{Q}_{i}(u) < 0,
    \end{cases}
\end{align}
and
\begin{equation} \label{tilde_rrro_2}
\displaystyle \bar{\tilde{\rho}}(v)_{i+\frac{1}{2}} =
\begin{cases}
\frac{ G(\overline{\rho}_{i},\overline{\rho}_{i+1},v_{i+\frac{1}{2}}) - G(\overline{\rho}_{i},\overline{\rho}_{i+1},0) }{ v_{i+\frac{1}{2}} }  \;\; \textnormal{ if } v_{i+\frac{1}{2}} \neq 0,\\
        \bar{\rho}_{i+1} -(\bar{\rho}_{i+1}-\bar{\rho}_{i})g'(0) \;\;\;\;\;\;\;\textnormal{ if } v_{i+\frac{1}{2}} = 0.
    \end{cases}
\end{equation}
We also use a  Newton-method with an error bound fixed to the zero machine.
\end{itemize}
Once these three steps are accomplished we consider that we have computed (an approximation of) $\textbf{T}^{n}(u)$ for a given $u \in X(\mathcal{T}^{\star}).$
We thus use a Picard-iteration scheme which consists in the sequence $(u_{k}^{n})_{k \in \N}$ defined by induction as follows:
\begin{equation} \label{picard-iterates}
\begin{cases}
u_{0}^{n} = u^{n},\\
u_{k+1}^{n} = \textbf{T}^{n}(u_{k}^{n}), \quad k \in \N.
\end{cases}
\end{equation}
If the sequence $(u_{k}^{n})_{k \in \N}$ converges to some $u_{\star}^{n} \in X(\mathcal{T}^{\star})$ then, since $\textbf{T}^{n}$ is a continuous map, the limit verifies $\textbf{T}^{n}(u_{\star}^{n}) = u_{\star}^{n}$ which exactly means that $u_{\star}^{n}$ is a solution of the non linear system \eqref{EPBD}. Our stopping criterion for the algorithm is (provided $u_{k}^{n} \neq 0$ for each $k$ and the sequence converges to some $u_{\star}^{n} \neq 0$):
\begin{itemize}
    \item Compute:
$$N_{\star} := \inf \Big \lbrace k \in \N \: : \frac{\|u_{k+1}^{n} -u_{k}^{n}\|_{L^{\infty}(\T)}}{\|u_{k}^{n}\|_{L^{\infty}(\T)}} \leq 10^{-7} \Big\rbrace.$$
Since the sequence is assumed to converge towards a non zero limit this number is well-defined.
\item Update the approximate solution by setting:  \begin{align}
 & u^{n+1} = u_{N_{\star}}^{n},\\ 
  &\rho^{n+1} = \bar{\rho}(u_{N_{\star}}),\\
  &\phi^{n+1} = \varphi(\rho^{n+1}).
\end{align}
Of course, this is an approximation of a fixed point up to the threshold error. 
\end{itemize}
In the following numerical experiment we fix the function $g$ to be given by
\begin{equation}\label{def_g_num}
g(u) = \begin{cases}
 u \;\;\;\;\;\;\;\;\;\;\;\;\textnormal{ if } u \geq \Delta x,\\
 \frac{(u+\Delta x)^2}{4 \Delta x} \;\; \textnormal{ if } -\Delta x < u < \Delta x,\\
 0 \;\;\;\;\;\;\;\;\;\;\;\; \;\textnormal{ if } u \leq -\Delta x.
\end{cases}
\end{equation}

\subsection{Non linear stability around constant states}
We consider a constant state of the form
\begin{align}
    \begin{cases}
        \bar{ u } \in \R,\\
        \bar{\phi} = 0,\\
        \bar{\rho} = e^{-\bar{\phi}} = 1.
    \end{cases}
\end{align}
This constant state is clearly a stationary solution of both  \eqref{EPB} and \eqref{EPB0}. 
We consider a fluctuation around the constant state $(\bar{\rho},\bar{u})$ in the form 
\begin{align}
    \rho_{\varepsilon}^{\textnormal{ini}}(x) - \bar{\rho}=  \frac{\varepsilon^{s}}{2} \sin\Big(2\pi x \lfloor \varepsilon^{-1} \rfloor \Big), \quad u_{\varepsilon}^{\textnormal{ini}}(x)  - \bar{u} =  \varepsilon \sin(2\pi x), \quad x \in [0,1], \label{initial_fluctuations}
\end{align}
where $\varepsilon \in (0,1]$ and $s \geq 0.$ In particular, we see that the fluctuation around $\bar{\rho}$ oscillates at the spatial scale $\varepsilon$. Observe besides  that for $0 < s' < s$ we have
\begin{align} \label{quasi_neutral_initialy}
\| \rho_{\varepsilon}^{\textnormal{ini}} - \bar{\rho} \|_{H^{s'}(\T)} \longrightarrow 0 \text{  as  } \varepsilon \longrightarrow 0.
\end{align}
In \cite{puguo}, it is shown
that provided $s$ is large enough \eqref{quasi_neutral_initialy} is propagated on $[0,T]$ for a Sobolev exponent which is smaller than $s'$. The modulated energy estimate \eqref{modulated_energy_estimate} enables to show the convergence in $L^{2}(\T)$ on $[0,T]$ provided the initial data is such that $\mathcal{E}(0) \longrightarrow 0$ as $\varepsilon \rightarrow 0$ (see \cite{Han-Kwan01082011}). Its discrete analogue  is \eqref{dis_mod_energy_estimate} for constant states with $\bar{u} = 0.$ The initial data is discretized in a finite volume manner, that is  
\begin{align}
    &\rho^{0}_{i} - \bar{\rho}  =  \frac{\varepsilon^{s}}{4 \pi \Delta x \lfloor \varepsilon ^{-1} \rfloor}\big( \cos\Big( 2\pi x_{i-\frac{1}{2}}\lfloor \varepsilon^{-1} \rfloor \Big) - \cos\Big( 2\pi x_{i+\frac{1}{2}}\lfloor \varepsilon^{-1} \rfloor \Big) \Big), \quad i \in \lbrace 0,..,N\rbrace \label{ini_1}\\
   & u^{0}_{i} - \bar{u} = \frac{\varepsilon}{2\pi \Delta x} \big( \cos(2\pi x_{i-\frac{1}{2}}) - \cos(2\pi x_{i+\frac{1}{2}})\big), \quad i \in \lbrace 0,..,N\rbrace.\label{ini_2} 
\end{align}
The initial potential $- \phi^{0}$ satisfies the discrete nonlinear Poisson equation: 
\begin{align*}
    \varepsilon^2 (\Delta \phi^{0})_{i} + e^{-\phi^{0}_{i}} = \rho_{i}^{0}, \quad i \in \lbrace 0,...,N\rbrace.
\end{align*}
\subsubsection{ Convergence as $\varepsilon \rightarrow 0$ with a fixed mesh-size}

\paragraph{ a) A well-prepared data on a coarse mesh.}
The mesh size is $\Delta x = 10^{-2}$ and the time step is $\Delta t = \frac{1}{2}\Delta x.$ The final time is $T = 1000  \Delta t.$  
In Table \ref{tab:mod_energy_table_1}, we report the values of the modulated energy at initial and final time for different values of $\varepsilon$ for a initial data of the form \eqref{ini_1}-\eqref{ini_2} with $s =1$.
\begin{table}[h]
    \centering
    \begin{tabular}{|c|c|c|c|}
        \hline
        $\varepsilon$ & $\mathcal{E}(\rho^{N_{T}},u^{N_{T}},\phi^{N_{T}}| \bar{u}, \bar{\phi})$&  $\mathcal{E}(\rho^{0},u^{0},\phi^{0}| \bar{u}, \bar{\phi})$\\ \hline
        0.1 &  0.00147172  &  0.0225411 \\ \hline
        0.01 &  1.61869e-05  &  0.000224353  \\ \hline
        0.001 &  1.61757e-07  &  2.24352e-06 \\ \hline
        0.0001 &  1.61787e-09  &  2.24348e-08 \\ \hline
    \end{tabular}
    \caption{Modulated energy at final and initial time for different values of $\varepsilon$ for an initial data of the form \eqref{ini_1}-\eqref{ini_2} with $s=1.$}
    \label{tab:mod_energy_table_1}
\end{table}

We observe that whatever the value of $\varepsilon$ is, the modulated energy at final time is lower than the modulated energy at initial time. It is an expected behavior of our scheme. Moreover, we see that when $\varepsilon$ decreases to zero, the modulated energy also decreases towards zero. We measure a convergence rate towards zero in $\varepsilon$ which is $\mathcal{O}(\varepsilon^{2})$. It is exactly the same rate as the rate of decrease towards zero of the modulated energy at initial time. During the simulation,  we have checked the total energy decay, the mass conservation and the conservation of positivity of the density. Note that the time step and the mesh size are fixed and completely independent of $\varepsilon$. These results are in good agreement with Theorem \ref{main_result_1} and the item a) of Theorem \ref{main_result_2}.
We have performed the same test with $\bar{u} \in \lbrace -4,-2,2,4\rbrace$ and we have obtained comparable results. These results illustrate the unconditional stability of our scheme. However, the results  must be interpreted with care since it is only a rough illustration of the convergence as $\varepsilon \rightarrow 0$ on a coarse mesh. We do not claim that our scheme is accurate when $\varepsilon \rightarrow 0.$
\paragraph{b) Evaluation of the numerical dissipation rate when $\varepsilon$ is fixed.}
We quantify the the numerical dissipation when $\varepsilon$ is fixed but $\Delta x$ and $\Delta t$ tend both to zero. We expect that for smooth fluctuations, the numerical dissipation rate should tend towards zero since in the continuous case the energy of smooth solutions is conserved. 
The numerical parameters are: $\varepsilon = 10^{-1}$, $\Delta t = \frac{1}{2} \Delta x.$ The final time is $T = 0.2$.
The initial data is of the form \eqref{ini_1}-\eqref{ini_2} with $s = 1$. In Table \ref{tab:mod_energy_dx}, we report the numerical dissipation rate defined by
\begin{equation}
    \tau(\Delta x) = \frac{\log(\mathcal{E}(\rho^{N_{T}},u^{N_{T}},\phi^{N_{T}}| \bar{u}, \bar{\phi})) - \log(\mathcal{E}(\rho^{0},u^{0},\phi^{0}| \bar{u}, \bar{\phi}))}{T}.
\end{equation}
\begin{table}[h]
    \centering
    \begin{tabular}{|c|c|c|c|}
        \hline
        $\Delta x$ & $\tau(\Delta x)$\\ \hline
        0.01 &  -0.796237   \\ \hline
        0.005 &  -0.416073   \\ \hline
        0.0025 &  -0.222488 \\ \hline
    \end{tabular}
    \caption{Numerical dissipation rate for $\varepsilon = 10^{-1}$ for three values of $\Delta x$.}
    \label{tab:mod_energy_dx}
\end{table}
We see that with the CFL condition $\Delta t = \frac{1}{2} \Delta x$ the numerical dissipation rate is of order one in $\Delta x$. In Figure \ref{fig:mod_energy_dx}, we still consider $\Delta x \in \lbrace 0.01,0.005,0.0025 \rbrace$ and plot for each $\Delta x$ the time evolution of the modulated energy on $[0,T]$ (left part) and the density $\rho(T, \cdot)$ (right part). 
\begin{figure}
\centering
\begin{tabular}{cc}
\includegraphics[width=0.4\textwidth]{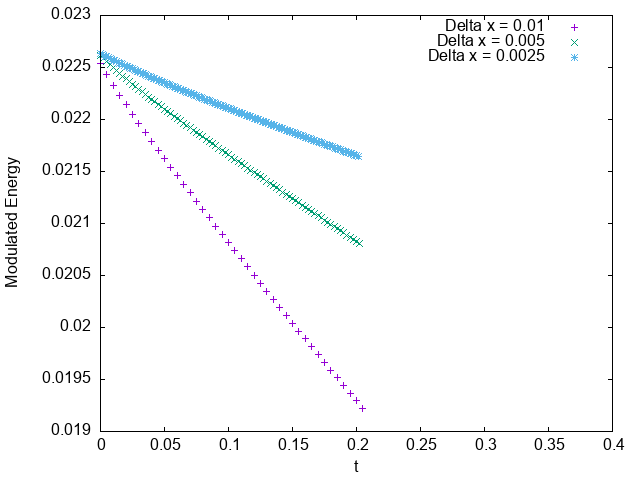} & \includegraphics[width=0.4\textwidth]{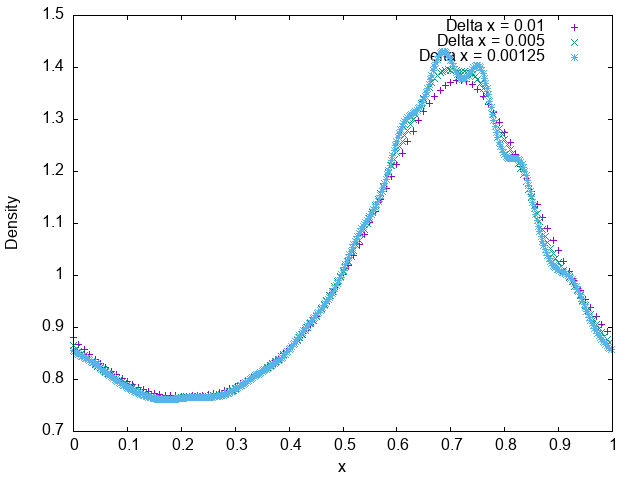}
\end{tabular}
    \caption{Left: time evolution of the modulated energy on $[0,T]$ with $\varepsilon = 10^{-1}$ for three values of $\Delta x$. Right: density $\rho(T=0.2,\cdot)$  with $\varepsilon = 10^{-1}$ for three values of $\Delta x$.}
    \label{fig:mod_energy_dx}
\end{figure}


\paragraph{c) A well-prepared data on a fine mesh.}
Here, we still consider the initial data of the form \eqref{ini_1}-\eqref{ini_2} with $s =1$. The mesh size is $\Delta x = 10^{-3}$ and the time step is $\Delta t = \frac{1}{2} \Delta x$. The final time is $T = 20 \Delta t$.  We report in Table \ref{tab:mod_energy_table_3} the values of the modulated energy at initial and final time for different values of $\varepsilon$. We expect the scheme to converge in $L^{\infty}([0;T]; L^{2}(\T))$ but a priori not in $L^{\infty}([0;T];H^{1}(\T))$  since there is a loss of one power of $\varepsilon$ when we differentiate \eqref{initial_fluctuations}. 
\begin{table}[h]
    \centering
    \begin{tabular}{|c|c|c|c|}
        \hline
        $\varepsilon$ & $\mathcal{E}(\rho^{N_{T}},u^{N_{T}},\phi^{N_{T}}| \bar{u}, \bar{\phi})$&  $\mathcal{E}(\rho^{0},u^{0},\phi^{0}| \bar{u}, \bar{\phi})$\\ \hline
        0.1 &  0.0225142  &  0.022534 \\ \hline
        0.05 &  0.00562696  &  0.00563343 \\ \hline
        0.025 &  0.00140584  &  0.00140829 \\  \hline
        0.0125 & 0.000353224  &  0.00035338 \\ \hline 
    \end{tabular}
    \caption{Modulated energy at final and initial time for different values of $\varepsilon$ for an initial data of the form \eqref{ini_1}-\eqref{ini_2} with $s=1$.}
    \label{tab:mod_energy_table_3}
\end{table}
We see that the modulated energy still decays as $\varepsilon$ decays towards zero. The order of convergence is $\mathcal{O}(\varepsilon^{2})$. 
In Figures \ref{fig:density-eps0.1}, 
we represent the initial density and the final density on the refined mesh for the values of $\varepsilon$ given in Table \ref{tab:mod_energy_table_3}.  We see that the spatial oscillations are still present and it seems that there is also an oscillatory behavior in time. To investigate the time oscillations, we plot in Figure \ref{rhomrhob} the time evolution of $\varepsilon^{-2}\|\rho(t, \cdot)-\bar{\rho}\|_{L^2(\T)}$ (the rescaling by $\varepsilon^2$ is needed to get comparable amplitudes). The oscillation period does not seem to depend strongly on $\varepsilon$. This behavior has already been observed in \cite{degond_b} thanks to a linear stability analysis. The main reason is that the zero order term in the Poisson equation tends to stabilize the high spatial frequency mode for the electric potential.

\begin{figure}
\centering
\begin{tabular}{ccc}
\includegraphics[width=0.32\textwidth]{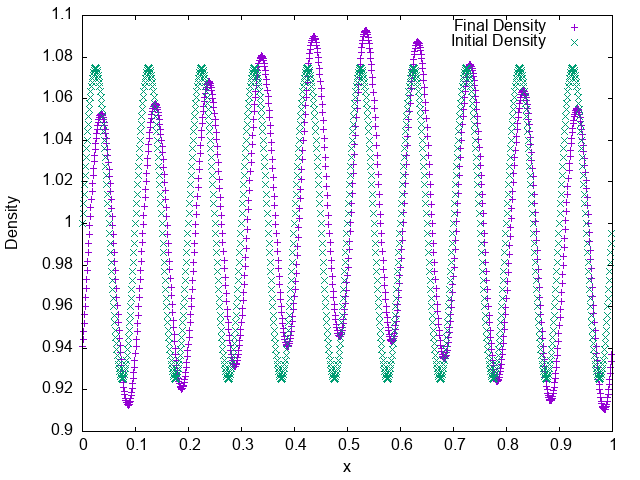}& \!\!\!\!\!\!\!\!\includegraphics[width=0.32\textwidth]{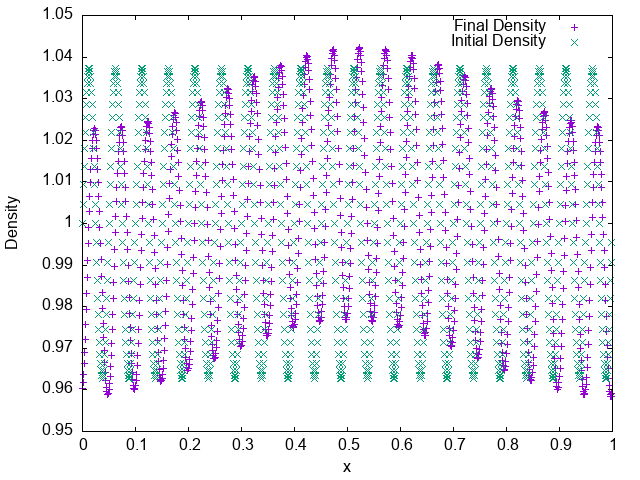} &   \!\!\!\!\!\!\!\! \includegraphics[width=0.32\textwidth]{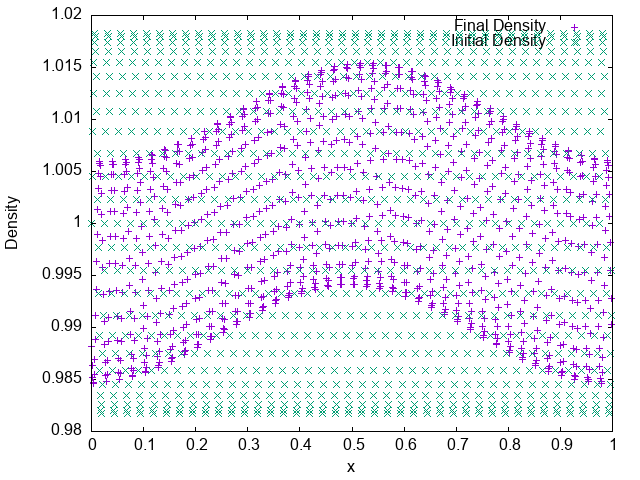}
\end{tabular}
    \caption{Plots of the initial density $\rho(0, \cdot)$ and final density $\rho(20\Delta t, \cdot)$ for $\varepsilon \in \lbrace 0.1, 0.05, 0.025\rbrace$ on the fine mesh: $\Delta x=10^{-3}$ and $\Delta t=\frac{1}{2}\Delta x$. }
    \label{fig:density-eps0.1}
\end{figure}

\begin{figure}
\centering
\includegraphics[width=0.6\textwidth]{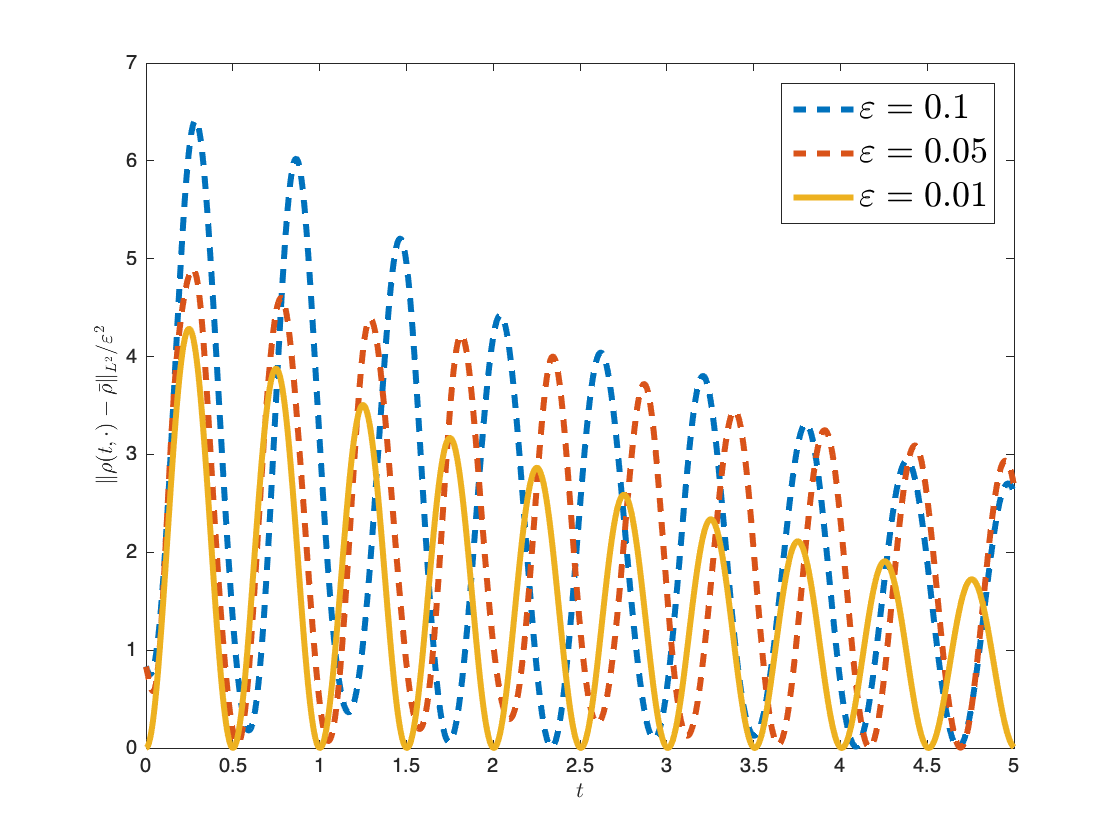}
    \caption{Time evolution of $t\in [0, 5]\mapsto \|\rho(t, \cdot)-\bar{\rho}\|_{L^2}/\varepsilon^2$ for $\varepsilon \in \lbrace 0.01, 0.05, 0.1\rbrace$.  $\Delta x=1/200$ and $\Delta t=\frac{1}{2}\Delta x$ }
    \label{rhomrhob}
    \end{figure}
\subsection{ An ill-prepared data on a fine mesh.}
We now consider the following initial data 
\begin{align}
    \rho_{\varepsilon}^{\textnormal{ini}}(x) - \bar{\rho}=  \frac{1}{2} \sin\Big(2\pi x \lfloor \varepsilon^{-1} \rfloor \Big), \quad u_{\varepsilon}^{\textnormal{ini}}(x)  - \bar{u} =  \sin(2\pi x), \quad x \in [0,1], \label{initial_fluctuations_2}
\end{align}
which is actually not well prepared. The mesh size is $\Delta x = 10^{-3}$, the time step is $\Delta t = \frac{1}{2} \Delta x$ and the final time is $T = 20 \Delta t.$ 
We report on Table \ref{tab:mod_energy_table_4} the values of the modulated energy at initial and final time for three values of $\varepsilon$. 
\begin{table}[h]
    \centering
    \begin{tabular}{|c|c|c|c|}
        \hline
        $\varepsilon$ & $\mathcal{E}(\rho^{N_{T}},u^{N_{T}},\phi^{N_{T}}| \bar{u}, \bar{\phi})$&  $\mathcal{E}(\rho^{0},u^{0},\phi^{0}| \bar{u}, \bar{\phi})$\\ \hline
        0.1 &  2.24793  &  2.2534 \\ \hline
        0.05 &  2.24739  &  2.25337 \\ \hline
        0.025 &  2.2465  &  2.25326 \\ \hline
    \end{tabular}
    \caption{Modulated energy at final and initial time for different values of $\varepsilon$ for an ill prepared initial data of the form \eqref{initial_fluctuations_2}.}
    \label{tab:mod_energy_table_4}
\end{table}
In figures \ref{fig:density_ill-eps0.1},  
we represent the density at initial and final time for three different values of $\varepsilon.$

\begin{figure}
\centering
\begin{tabular}{ccc}
\includegraphics[width=0.3\textwidth]{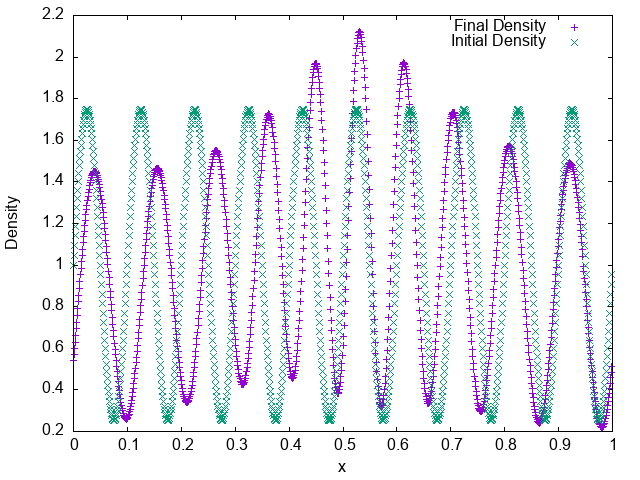} & \!\!\!\!\!\!\!\!\includegraphics[width=0.3\textwidth]{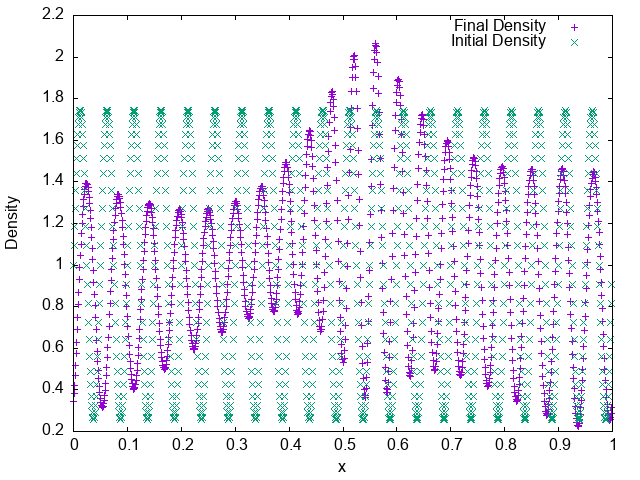} & \!\!\!\!\!\!\!\!\includegraphics[width=0.3\textwidth]{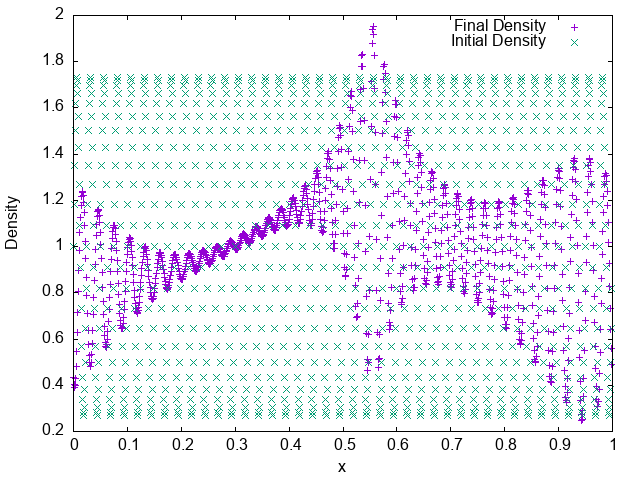} 
\end{tabular}
 \caption{Plots of the initial density $\rho(0, x)$ and final density $\rho(20\Delta t, x)$ for $\varepsilon \in \lbrace 0.1, 0.05, 0.025\rbrace$ on a fine mesh: $\Delta x=10^{-3}$ and $\Delta t=\frac{1}{2}\Delta x$. }
    \label{fig:density_ill-eps0.1}
\end{figure}


\subsection{ Five-branch solution}
This test case is inspired from \cite{fiveb, degond_b, arun} but here we consider (because of the periodic boundary conditions), the following initial condition 
\begin{eqnarray*}
\rho^{\textnormal{ini}}(x) &=& 
\begin{cases}
0.1 + \exp\Big(\frac{0.1}{(x-3\pi/4)(x-5\pi/4)}\Big)  \mbox{ if  $\frac{3\pi}{4} < x <  \frac{5\pi}{4},$ }\\
0.1 \mbox{ if $x \in [0,2\pi] \setminus (\frac{3\pi}{4}, \frac{5\pi}{4}).$}
\end{cases}
\\
    u^{\textnormal{ini}}(x) &=& \sin^3(x), \quad x \in [0,2\pi].
\end{eqnarray*}
Note that $\rho^{\textnormal{ini}} \in \mathcal{C}^{\infty}_{c}((0,2\pi))$ so it has a smooth $2\pi$-periodic extension. The electric potential $\phi^{\textnormal{ini}}$ is computed from the nonlinear Poisson equation. 
The space domain is $2\pi \T \equiv  [0, 2\pi)$ discretized with a mesh size $\Delta x=2\pi/400$ and the time step is $\Delta t= \Delta x/2$. We set the final time $T = 0.5$. 
We run the method presented before with the limit scheme ($\varepsilon = 0$) which consists in replacing the Poisson equation by the algebraic relation $\phi^{n+1} = -\log(\rho^{n+1})$. Moreover, we also present some results obtained by a numerical method for the limit model \eqref{EPB0} based on the same spatial and temporal discretization. The main difference lies in the fact the discretization we use for \eqref{EPB0} is conservative which is not the case for the limit scheme ($\varepsilon = 0$).  The numerical parameters are the same for the three solvers. 

In Figures \ref{fiveb_epsm4} and \ref{fiveb},  
we plot the space dependency of the density and velocity at time $t=0.5$ for $\varepsilon \in \lbrace 10^{-4}, 10^{-2}, 0\rbrace$. The case $\varepsilon = 0$ corresponds to the limit scheme. 
One can observe the scheme for $\varepsilon=10^{-4}$ and $\varepsilon=0$ are almost indistinguishible whereas oscillations are present for $\varepsilon=10^{-2}$.  
In Figure \ref{fiveb}, the density is plotted at time $t=0.5$ for $\varepsilon \in \lbrace 10^{-4}, 0\rbrace$ as before but we also add the result obtained by the scheme for \eqref{EPB0}. In the region where the solution is smooth, 
the three curves are very similar whereas some differences can be observed around the discontinuity $x\approx 1.9$. Indeed, the propagation of the discontinuity seems to be different due to the non conservative treatment of the 
term $\rho\partial_x\phi$ (see the inset in Figure \ref{fiveb}).

\begin{figure}
\centering
\includegraphics[width=0.4\textwidth]{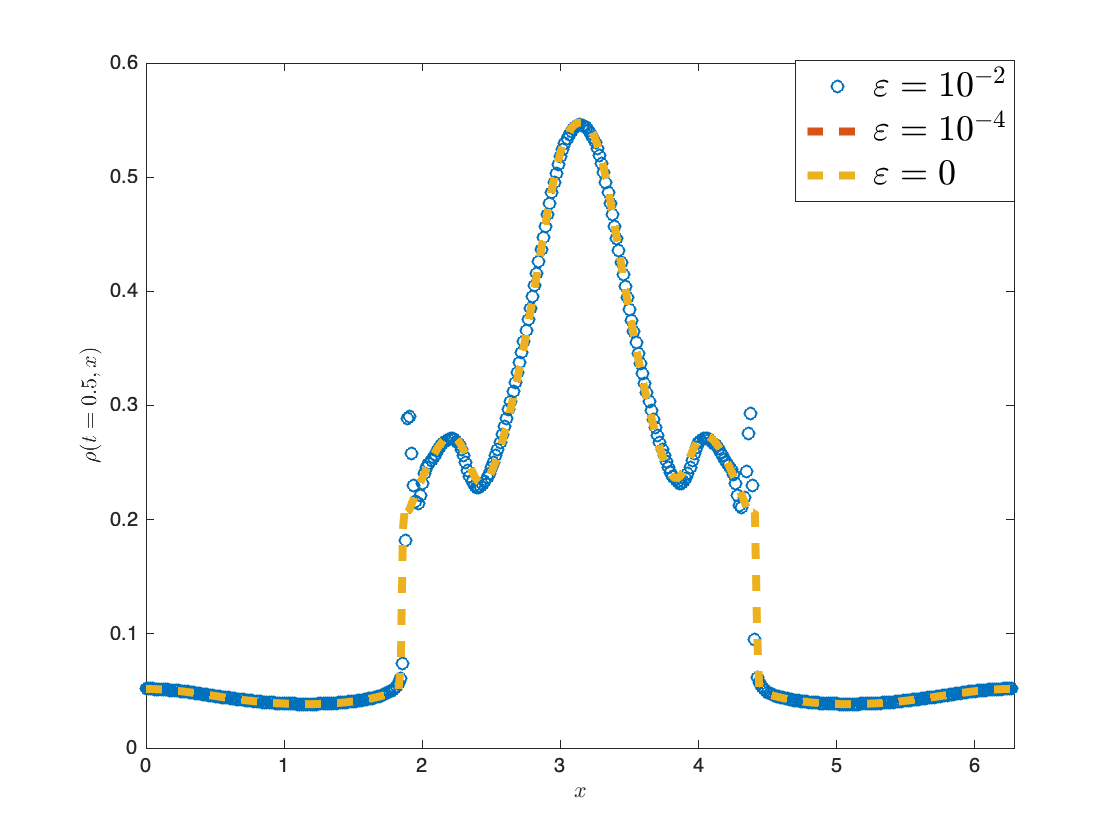}
\includegraphics[width=0.4\textwidth]{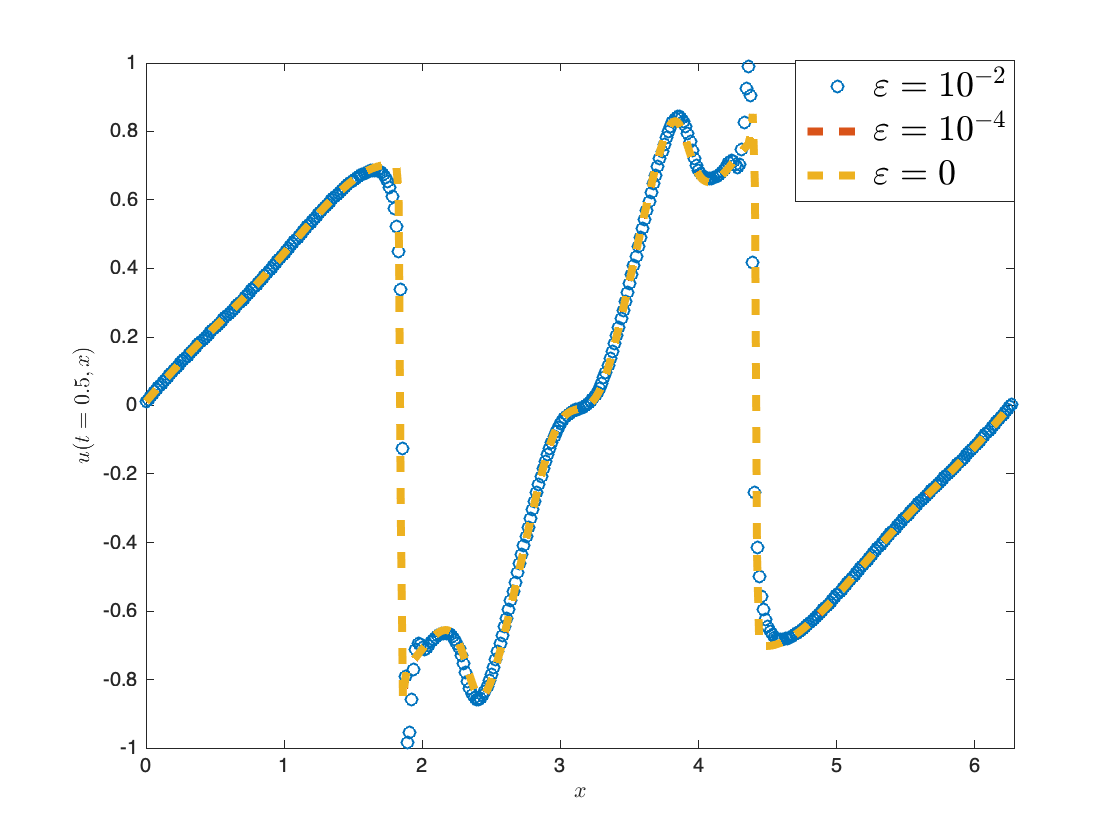}
    \caption{Five-branch test: comparison of the scheme for $\varepsilon=10^{-2}, 10^{-4}$ and the asymptotic scheme $\varepsilon=0$. Left: density $\rho(t=0.5, \cdot)$. Right: velocity $u(t=0.5, \cdot)$. $\Delta x=2\pi/400$ and $\Delta t=\frac{1}{2}\Delta x$.} 
    \label{fiveb_epsm4}
\end{figure}

\begin{figure}
\centering
\includegraphics[width=0.6\textwidth]{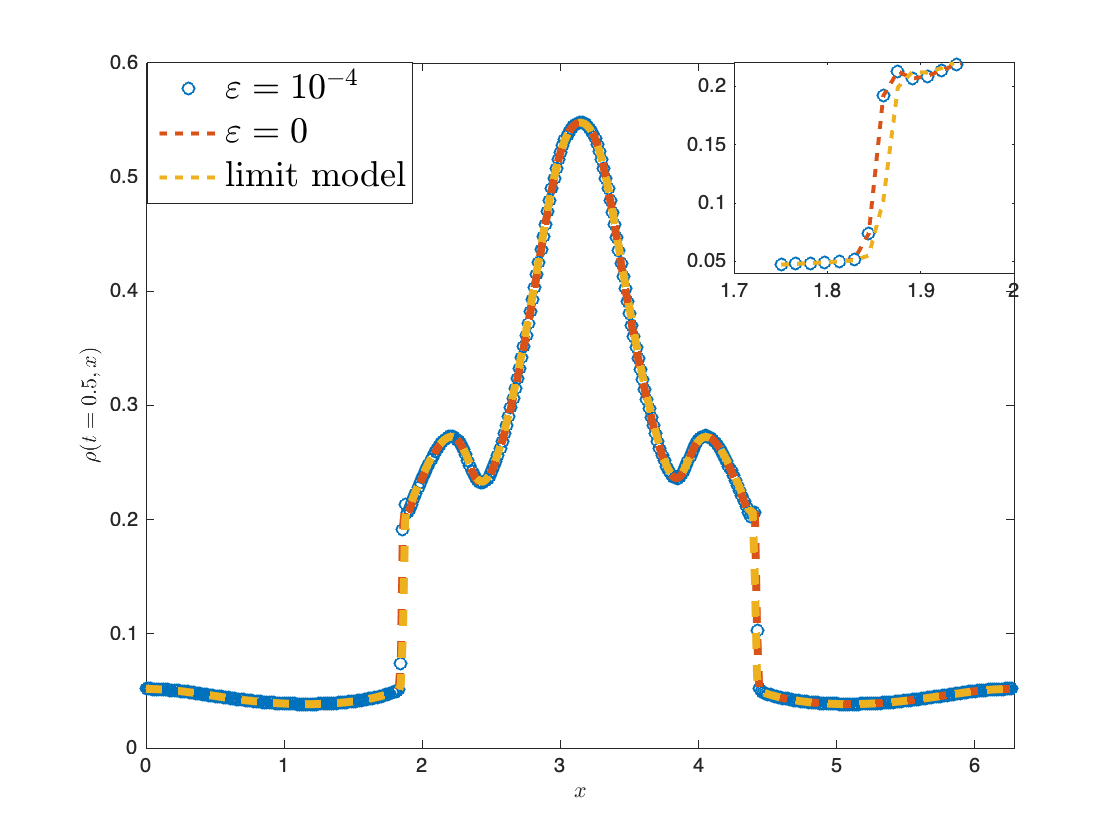}
 \caption{Five-branch test: comparison of the scheme for $\varepsilon=10^{-4}$, the asymptotic scheme $\varepsilon=0$ and the limit model. Density $\rho(t=0.5, x)$. The inset is a zoom around $x=1.9$. $\Delta x=2\pi/400$ and $\Delta t=\frac{1}{2}\Delta x$. }
    \label{fiveb}
\end{figure}

\paragraph{Acknowledgement.} MB acknowledges the financial support from Inria Rennes. He is also grateful to Institut de Recherche Mathématiques de Rennes and Ecole Normale Supérieure de Rennes for the material support during the elaboration of this work. This work has been carried out within the framework of the EUROfusion Consortium, funded by the European Union via the Euratom Research and Training Programme (grant agreement No 101052200 EUROfusion). Views and opinions expressed are however those of the authors only and do not necessarily reflect those of the European Union or the European Commission. The work has been supported by the French Federation for Magnetic Fusion Studies (FR-FCM). 
\bibliographystyle{abbrv}
\bibliography{biblio}

\end{document}